\documentclass[12pt]{amsart}

\usepackage{geometry}
\usepackage{amsfonts}
\usepackage{amssymb}
\usepackage{calrsfs}
\usepackage[mathscr]{eucal}
\usepackage{amsmath}
\usepackage{amsthm}
\usepackage[all]{xy}
\usepackage[mathscr]{eucal}
\usepackage{graphicx}
\usepackage{longtable}
\usepackage{wrapfig}
\usepackage{algpseudocode}
\usepackage{algorithm}

\usepackage{color}
\definecolor{red-}{rgb}{1.0,0.0,0.0}
\definecolor{green-}{rgb}{0.0,0.7,0.0}
\definecolor{brown-}{rgb}{0.9,0.6,0.0}

\usepackage{hyperref}
\hypersetup{
    colorlinks=true,
    linkcolor=blue,
    citecolor=blue,
    filecolor=blue,
    urlcolor=blue
}

\newtheorem{defi}{Definition}[section]
\newtheorem{ex}[defi]{Example}
\newtheorem{thm}[defi]{Theorem}
\newtheorem{notat}[defi]{Notation}
\newtheorem{cor}[defi]{Corollary}
\newtheorem{prop}[defi]{Proposition}
\newtheorem{lem}[defi]{Lemma}
\newtheorem{rem}[defi]{Remark}

\begin{document}

\title{Hilbert curve characterizations of some relevant polarized manifolds}

\author{Antonio Lanteri and Andrea Luigi Tironi}

\date{\today}

\address{Dipartimento di Matematica ``F. Enriques'',
Universit\`a degli Studi di Milano, Via C. Saldini, 50,  I-20133 Milano,
Italy}\email{antonio.lanter@unimi.it}
\address{Departamento de Matem\'atica, Universidad de Concepci\'on, Barrio Universitario s/n, Casilla
160-C, Concepci\'on, Chile} \email{atironi@udec.cl}


\begin{abstract}
Hilbert curves of special varieties
like Fano manifolds of low coindex as well as fibrations having
such a manifold as general fiber, endowed with appropriate polarizations, are investigated. In particular,
all most relevant varieties arising in adjunction theory are
characterized in terms of their Hilbert curves.
\end{abstract}

\maketitle

\section*{Introduction}\label{Intro}

The Hilbert curve $\Gamma=\Gamma_{(X,L)}$ of a polarized manifold $(X,L)$ was introduced in \cite{BLS} and further studied in
\cite{L1}, \cite{L2}. It is the affine plane curve of degree $n=\dim X$ defined by $p(x,y)=0$, where
$p$ is the complexified of the polynomial
expression provided by the Riemann--Roch theorem for the Euler--Poincar\'e
characteristic $\chi(xK_X+yL)$, regarding $x$ and $y$ as complex variables. Clearly $p \in \mathbb Q[x,y]$ is
a numerical polynomial.
As shown in \cite{BLS}, $\Gamma$ encodes interesting properties of the pair $(X,L)$; in particular
it is sensitive to the possibility of fibering $X$ over a variety of smaller dimension via an adjoint bundle to $L$.
This makes polarized varieties arising in adjunction theory \cite{BS} very interesting from the point of view of
their Hilbert curves. In this paper, inspired by the study of Hilbert curves of projective bundles over a smooth curve
made in \cite{L1}, we provide a unifying perspective of the Hilbert curves of these special varieties.

Since Fano manifolds are the building blocks of these varieties,
we first address (Section 2) the study of pairs $(X,L)$, where $X$
is a Fano manifold of dimension $n$ and $L= \frac{r}{\iota_X}
(-K_X)$, $\iota_X$ being the index of $X$ and $r$ any positive integer.
For such a pair $(X,L)$ we
determine explicitly the canonical equation $p(x,y)=0$ of $\Gamma$
(Algorithm \ref{Alg1}). It turns out that, in $\mathbb A^2_{\mathbb C}$,
$\Gamma$ consists of $n$ parallel lines with slope
$\frac{\iota_X}{r}$. While $\iota_X-1$ of these lines are always defined
over $\mathbb{Q}$, the total reducibility of $p(x,y)$ over
$\mathbb{Q}$ is a delicate problem for $\iota_X\leq n-1$. We provide a partial answer
concerning toric Fano manifolds (Proposition \ref{prop2} and Tables \ref{types n=2,3} and \ref{types n=4})
and a complete discussion for del Pezzo manifolds
(Table \ref{the geography of HC}).
Moreover for $r$ and $\iota_X$ coprime and $\iota_X > \frac{n+1}{2}$ we
provide a characterization of pairs $(X,L)$ as above in terms of
their Hilbert curves (Corollary \ref{Fano2}). The above applies to Fano
manifolds of low coindex, including in particular the cases in
which $X$ is a projective space $\mathbb P^n$, a smooth quadric
hypersurface $\mathbb Q^n$, a del Pezzo or a Mukai manifold (Theorem \ref{projective space,quadric,del Pezzo} and Proposition \ref{Mukai}). We
want to emphasize that in general $(X,L)$ is characterized by various properties
of $\Gamma$ and not only by its shape. For instance,
the Hilbert curve of $\big(\mathbb P^3, \mathcal O(2)\big)$
and that of the del Pezzo threefold of degree $7$
consist of three parallel evenly spaced lines with the same slope
(over $\mathbb C$ or $\mathbb R$);
what makes them different is the arithmetic (Remark \ref{analog-conjecture}).

Next we consider Fano fibrations of low coindex. Here we can assume
that $\text{\rm{rk}}\langle K_X,L \rangle = 2$.
When $X$ fibers over a curve and $L_F = \frac{r}{\iota_F}(-K_F)$, $F$ being a general fiber (Section 3),
some ideas used in \cite{L1} to deal with the case of projective bundles are
further developed and lead to a complete characterization of $\mathbb P$-bundles (Theorem \ref{thm2}),
$\mathbb{Q}$-fibrations (Theorem \ref{thm3}) and del Pezzo fibrations (Theorem \ref{expr_with_r=1}) in terms of their Hilbert curves,
assuming that a suitable adjoint bundle is nef.
In particular, this generalizes \cite[Theorem 4.1]{L1} and \cite[Theorem 6]{L2}.
In fact, to get the canonical equation of $\Gamma$ we describe two approaches. The first one relies on a technical lemma (Lemma \ref{P-fibrations}),
which, under certain conditions, ensures that there exists an ample line bundle $\mathcal{L}$ on $X$, deriving from
$L$ and $K_X$, inducing the fundamental divisor on the general fiber: in a sense this allows us to work as if it were $r=1$. We illustrate this approach
for $\mathbb{Q}$-fibrations relying on \cite[Proposition 3]{L2}. A second approach, which is necessary in the general case due to the lack of specific results
for $r=1$ (e.\ g. for del Pezzo fibrations), is based on the additivity
of the Euler--Poincar\'e characteristic $\chi$ for exact sequences.
Essentially, this allows us to relate the
equation of the Hilbert curve of a fibration with that of its general fiber and since
the latter is a Fano manifold of low coindex we can apply the results in Section 2.
This leads to an algorithm (Algorithm 2) to obtain the equation of the Hilbert curve and we
make it explicit for i) $\mathbb{P}$-bundles and ii) del Pezzo fibrations.
As to case i), a conjecture \cite[Conjecture C$(n,r)$]{L1} claims
that a polarized manifold $(X,L)$ of dimension $n \geq 3$, with $\text{\rm{rk}}\langle K_X,L \rangle = 2$
is a $\mathbb P^{n-1}$-bundle over a smooth curve with $L$ inducing
$\mathcal O_{\mathbb P^{n-1}}(r)$ on every fiber, with $\mathrm{gcd}(r,n)=1$, if and only if its
Hilbert curve $\Gamma$ contains the fixed point
of the Serre involution and consists of $n$ lines, $n-1$ of which exactly are parallel each other, with
slope $\frac{n}{r}$ and evenly spaced. For $r=1$, C$(n,r)$ was proved in \cite{L1}.
As a consequence of Theorem \ref{thm2} it turns out that this
conjecture is true also for $r\geq 2$, provided that $rK_X+nL$ is nef.

More generally, in Section 4 we consider Fano fibrations of low coindex over a normal variety
of dimension $\geq 2$.
The technique relying on the additivity of $\chi$ for exact sequences applies also in
this case. Here we relate the equation of the Hilbert curve of a fibration with that of a suitable divisor, which is in turn
a Fano fibration of the same coindex but of smaller dimension and by induction we reduce
to the case of fibrations over a smooth curve, which allows us to apply the results in Section \ref{subsection1}.
This leads to an algorithm (Algorithm 3) to obtain the equation of the Hilbert curve, for instance for a projective bundle
over a smooth surface. To give a concrete example, we make it explicit for scrolls (Example \ref{ex scroll over S}); for a different approach
relying on direct Chern class computations we refer to \cite{L3}.
Moreover, also in this case we succeed to characterize the structure of $(X,L)$ in terms of its
Hilbert curve, under the assumption that $K_X+(n-1)L$ is nef (Theorem \ref{thm scrolls over S}).

Finally, the link to a program in MAGMA for checking the total reducibility of $p(x,y)$
over $\mathbb Q$ for toric Fano manifolds, as well as
the above algorithms produced in the paper, are contained in the
Appendix. Several computations have been done with the help of Maple $15$.

\section{Background material}\label{backgr}

 Varieties considered in this paper are defined over
 the field $\mathbb C$ of complex numbers. We use the standard notation
 and terminology
 from algebraic geometry. A manifold is any smooth projective variety.
 Tensor products of line bundles are denoted
 additively. The pullback of a vector bundle $\mathcal F$ on a manifold $X$
 by an embedding $Y \hookrightarrow X$ is simply denoted by
 $\mathcal F_Y$.
 We denote by $K_X$ the canonical bundle of a manifold $X$.
 The symbol $\equiv$ will stand for numerical equivalence.
 A \textit{polarized manifold}
 is a pair $(X,L)$ consisting of
 a manifold $X$ and an ample
 line bundle $L$ on $X$.

 A \textit{Fano manifold} is a manifold $X$ such that $-K_X$ is an ample line bundle ($X$ is also called
 a {\textit{del Pezzo surface}} if $\dim X=2$). The
 \textit{index} $\iota_X$ of $X$ is defined as the greatest positive integer which divides $-K_X$ in Pic$(X)$,
 the Picard group of $X$, while by the \textit{coindex} of $X$ we simply mean the
nonnegative integer $c_X:= \dim X +1-\iota_X$. Moreover, we say that a polarized
manifold $(X,L)$ of dimension $n$ is a \textit{del Pezzo manifold}
(respectively a \textit{Mukai manifold}) if
$K_X+(n-1)L=\mathcal{O}_X$ (respectively
$K_X+(n-2)L=\mathcal{O}_X$).

Let $(X,L)$ be a polarized manifold of dimension $n$; we say that $(X,L)$ is a \textit{Fano
fibration of coindex $n-m+1-t$} if there exists a surjective
morphism with connected fibers $\varphi:X\to Y$ onto a
normal variety $Y$ of dimension $m<n$ such that
$K_X+tL=\varphi^*H$ for some ample line bundle $H$ on $Y$ and
positive integer $t$. In particular, a \textit{scroll} $(X,L)$ is
a Fano fibration of coindex $0$, a \textit{quadric fibration} is
a Fano fibration of coindex $1$ and so on. Let us note here that
a Fano fibration of dimension $n$ and coindex $0$ over a curve
$C$ (or a surface $S$) is a projective bundle
$\mathbb{P}(\mathcal{V})$ for some ample vector bundle
$\mathcal{V}$ of rank $n$ (or $n-1$) over $C$ (or $S$). We say that a polarized manifold $(X,L)$ is a $\mathbb P$\textit{--bundle}
over a normal variety $Y$ if $X=\mathbb P(\mathcal F)$ for some
vector bundle $\mathcal F$ on $Y$ and $L$ is any ample line bundle on $X$;
we say that $(X,L)$ is a $\mathbb Q$\textit{--fibration} over $Y$ if $X$ is endowed with
a surjective morphism $X \to Y$ whose general fiber is a smooth quadric hypersurface
and $L$ is any ample line bundle on $X$.

 For the notion and the general properties of the \textit{Hilbert curve} associated to a
 polarized manifold we refer to \cite{BLS}. Here we just recall some basic facts.
 Let $(X,L)$ be a polarized manifold of dimension $n$.
 For any line bundle $D$ on $X$ consider the expression of the Euler--Poincar\'e  characteristic $\chi(D)$
 provided by the Riemann--Roch theorem
 \begin{equation}\label{rrh}
 \chi(D) = \frac{1}{n!} D^n - \frac{1}{2(n-1)!}K_XD^{n-1} + \text{\rm{terms of lower degree}} 
 \end{equation}
 (a polynomial of degree $n$ in the Chern class of $D$, whose coefficients are polynomials in the
 Chern classes of $X$ \cite[Theorem 20.3.2]{Hi}). Let
 $p$ (or $p_{(X,L)}$ to avoid possible ambiguity) be
 the complexified polynomial of $\chi(D)$, when we set $D = xK_X + yL$,
 with $x,y$ complex numbers, namely $p(x,y) := \chi (xK_X+yL)$.
 The Hilbert curve (HC for short) of $(X,L)$ is the complex affine plane curve $\Gamma = \Gamma_{(X,L)}$ of
 degree $n$ defined by $p(x,y)=0$ \cite[Section 2]{BLS}. We refer to $p(x,y)=0$
 as the {\it{canonical equation}} of $\Gamma$.
 Clearly, $p_{(X,L)}(x,y)=p_{(X,L')}(x,y)$ if $L \equiv L'$, hence two numerically equivalent
 polarizations on $X$ give rise to the same HC.
 If $\text{rk}\langle K_X, L \rangle = 2$ in $\text{Num}(X)$, and we consider $\text{N}(X):= \text{Num}(X) \otimes_{\mathbb Z} \mathbb C$ as a complex
 affine space, then $\Gamma$ is the section of the Hilbert variety of $X$ (\cite[$\S 2$]{BLS}) with the plane $\mathbb A^2= \mathbb C \langle K_X,L \rangle$,
 generated by the classes of $K_X$ and $L$. On the other hand, if $\text{rk}\langle K_X, L \rangle = 1$ in $\text{Num}(X)$, $\Gamma$
 loses this meaning, the plane of coordinates $(x,y)$ being only formal. We will refer to this situation as the {\it{degenerate case}}.
 Since $p \in \mathbb Q[x,y]$ is a numerical polynomial, $\Gamma$ is defined over $\mathbb Q$, hence
 we can also look at $\Gamma_{\mathbb R} \subset \mathbb A^2_{\mathbb R}$ and
 $\Gamma_{\mathbb Q} \subset \mathbb A^2_{\mathbb Q}$.

 Taking into account that $c:=\frac{1}{2}K_X$ is the fixed point of the Serre involution $D \mapsto K_X-D$ acting
 on $\text{N}(X)$, sometimes it is convenient to represent $\Gamma$ in terms of affine coordinates $(u=x - \frac{1}{2},
 v=y)$ centered at $c$ instead of $(x,y)$.
 In other words, we set $D=\frac{1}{2}K_X + E$, where $E=uK_X+vL$. Then
 $\Gamma$ can be represented with respect to these coordinates by $p(\frac{1}{2}+u,v)=0$.
 An obvious advantage is that, due to Serre duality, $\Gamma$ is invariant under
 the symmetry $(u,v) \mapsto (-u,-v)$.
 Sometimes, to deal with points at infinity, it is convenient to consider
 also the projective Hilbert curve $\overline{\Gamma} \subset \mathbb P^2$, namely the projective closure of $\Gamma$.
 In this case we use $x,y,z$ as homogeneous coordinates on $\mathbb P^2$, $z=0$ representing the line at infinity.
 Given a point $(x,y) \in \mathbb A^2$, we write $(x:y:1)$ to denote the same point when regarded as a point of $\mathbb P^2$.
 Moreover, we denote by $p_0(x,y,z)$ the homogeneous polynomial associated with $p(x,y)$
 (i.\ e., $p(x,y)=p_0(x,y,1)$), which defines the plane projective curve $\overline{\Gamma}$.
 Note that
 \begin{equation}\label{atinfty}
 p_0(x,y,0) = \frac{1}{n!} (xK_X+yL)^n
 \end{equation}
 in view of \eqref{rrh}. This will be used over and over. Another fact of frequent use will be the following.

 \begin{rem}\label{morphism}
 {\em Let $(X,L)$ be a polarized manifold of dimension $n\geq 3$ and suppose that $\sigma K_X+\tau L$ is nef and not big for some positive
 integers $\sigma, \tau$. Then there exists a morphism $\varphi:X\to Y$ onto a normal variety $Y$ with $\dim Y<n$ such that $\sigma K_X+\tau L=\varphi^*D$
 for a nef line bundle $D$ on $Y$. Actually, we can write $\sigma K_X+ \tau L = K_X+M$ where $M$ is an ample line bundle. This
is obvious for $\sigma =1$, while for $\sigma\geq 2$ we have
$$M= (\sigma-1)K_X + \tau L = \frac{\sigma-1}{\sigma} \big(\sigma K_X + \tau L \big) + \frac{\tau}{\sigma}\ L .$$
Thus $M$ is ample being the sum of a nef and an ample $\mathbb Q$-line bundles. Then by the Kawamata--Shokurov base-point free theorem
the linear system $|m\left(\sigma K_X+\tau L)\right)|$ is effective and base-point free for $m >> 0$.
Hence it defines a morphism $\Phi:X\to\mathbb{P}^N$, where the image has dimension $<n$, since $\sigma K_X+\tau L$ is not big.
The morphism $\varphi:X \to Y$ is defined by the Stein-factorization of $\Phi$. }
\end{rem}

 Finally let us discuss here the case $n=2$ as an example.
 So, let $(X,L)$ be a polarized surface; then its Hilbert curve is simply the conic $\Gamma: = \Gamma_{(X,L)} \subset \mathbb A^2_{\mathbb C}$, defined in coordinates $(u,v)$ by
 \begin{equation}\label{surf}
 p\left(\frac{1}{2}+u,v\right)= \frac{1}{2} \left( K_X^2 u^2 + 2K_X \cdot L uv + L^2 v^2 + 2\chi(\mathcal O_X)-\frac{1}{4}K_X^2 \right)= 0\ .
 \end{equation}
 By the Hodge index theorem we know that $K_X^2 L^2  - (K_X \cdot L) ^2 \leq 0$, with equality if and only if
 $K_X \equiv \lambda L$ (degenerate case).
 So $\Gamma$ is of parabolic type if and only if we are in the degenerate case.
 Now look at $\Gamma$ from the real point of view.
 The above expression is the quadratic orthogonal invariant of the conic $\Gamma_{\mathbb R}$, hence $\Gamma_{\mathbb R}$ is either a hyperbola or a couple of incident lines,
 except for the degenerate case, in which $\Gamma_{\mathbb R}$ is necessarily reducible: either a line with multiplicity $2$,
 a couple of parallel lines, or $\Gamma_{\mathbb R}=\emptyset$ (according to whether $K_X^2 - 8\chi(\mathcal O_X)$ is $=0$, $>0$ or $<0$ respectively).
 For instance, the last situation occurs for the cubic surface in $\mathbb P^3$.

 More generally, consider a del Pezzo surface $X$, set $-K_X=\iota_X H$, $d=H^2$, and let $L=rH$ for any positive integer $r$.
 The classification of del Pezzo surfaces implies the following facts for the HC, $\Gamma$, of the pair $(X,L)$.
 If $d \leq 7$, then $\Gamma_{\mathbb R} = \emptyset$. Actually, in this case, $\iota_X=1$ and $p\big(\frac{1}{2}+u,v\big)=\frac{d}{2}(u-rv)^2+ \big(1-\frac{d}{8}\big)$.
 Let $d=8$; then either: a) $X=\mathbb P^1 \times \mathbb P^1$, $\iota_X=2$, and then $L=\mathcal O_{\mathbb P^1 \times \mathbb P^1}(r,r)$, or
 b) $X = \mathbb F_1$ (the Segre--Hirzebruch surface of invariant $1$), $\iota_X=1$, and $L=r[2C_0+3f]$ ($C_0$ and $f$ being the $(-1)$-section ad a fiber
 respectively). Here
 \begin{equation} \notag
 p\left(\frac{1}{2}+u,v\right) = \begin{cases}
 (2u-rv)^2\quad \mathrm{in\ case\ a)} \\
 4(u-rv)^2 \quad \mathrm{in\ case\ b)}\ ,
 \end{cases}
 \end{equation}
 hence in both cases $\Gamma_{\mathbb R}$ is a line with multiplicity $2$.
 Finally, let $d=9$, then $X = \mathbb P^2$, $\iota_X=3$ and $L=\mathcal O_{\mathbb P^2}(r)$. In this case,
 $p\big(\frac{1}{2}+u,v\big) = \frac{1}{2}\big((3u-rv)^2-\frac{1}{4}\big)$
 hence $\Gamma_{\mathbb R}$ consists of two parallel lines.

\section{The case $\mathrm{rk}\langle K_X,L\rangle=1$: High index Fano manifolds}

Let $X$ be a Fano manifold of dimension $n\geq 2$ of index $\iota_X$. Then there exists an ample line bundle $H$ (a fundamental divisor) on $X$ such that $-K_X=\iota_XH$.
From now on our setting for Fano polarized manifolds $(X,L)$ will be the following:
\begin{equation}\label{Fano setting}
X\ \mathrm{is\ Fano\ with\ } n\geq 2\ \mathrm{and\ } L:=rH,\ \mathrm{where\ } H:=\frac{1}{\iota_X}(-K_X)\ \ \mathrm{is\ the\ fundamental\ divisor}.
\end{equation}
Let $(X,L)$ be as in \eqref{Fano setting}. Then $rK_X+\iota_XL=\mathcal{O}_X$, which implies $\mathrm{rk}\langle K_X,L\rangle=1$. Let $p(x,y)=0$ be the canonical equation
of the Hilbert curve $\Gamma_{(X,L)}$.
Recalling that $p(x,y)=\chi(xK_X+yL)$, we get
$$p(x,y)=\chi( (ry-\iota_Xx)H )=\chi( tH )=:q(t),$$
where $t:=ry-\iota_Xx$. Moreover, note that $q(t)=\chi( K_X+(t+\iota_X)H )=h^0(tH)$ for $t\geq 1-\iota_X$, by the Kodaira vanishing theorem.
Thus, if $-1\geq t\geq 1-\iota_X$ then $q(t)=0$. This shows that
$$q(t)=\varphi(t)\cdot\prod_{i=1}^{\iota_X-1} (t+i) \quad \mathrm{with}\ \deg\varphi(t)=c_X\ , \mathrm{the\ coindex\ of\ } X.$$
Set $\varphi(t):=\sum_{j=0}^{c_X}a_jt^j$ and observe that we need $c_X+1$ linearly independent linear conditions on $q(t)$ to determine the polynomial $\varphi(t)$.
So, for $s=0,1,...,c_X$ we see that
$\varphi(s)\cdot\prod_{i=1}^{\iota_X-1} (s+i)=q(s)=h^0(sH)$, i.\ e.,
$$a_{c_X}s^{c_X}+a_{c_X-1}s^{c_X-1}+ \dots +a_{1}s+a_0=\varphi(s)=\frac{h^0(sH)}{\prod_{i=1}^{\iota_X-1} (s+i)}\ .$$
This gives the following system of $c_X+1$ linear equations in the $c_X+1$ unknowns $a_0,a_1,...,a_{c_X}$
\begin{equation}\label{linear system Fano}
U\cdot
\left( \begin{array}{ c }
a_0 \\
a_1 \\
\vdots \\
a_{c_X} \\
\end{array}\right)=
\left( \begin{array}{ c }
\frac{h^0(\mathcal{O}_X)}{\delta(0)} \\
\frac{h^0(H)}{\delta(1)} \\
\vdots \\
\frac{h^0(c_X H)}{\delta(c_X)} \\
\end{array}\right)\ ,
\end{equation}
where $U$ is the $(c_X+1)\times (c_X+1)$ Vandermonde matrix
\begin{equation}\label{U}
U:=\left( \begin{array}{ c c c c c }
1& 0 & \cdots & 0 & 0 \\
1& 1 & \cdots & 1 & 1  \\
1& 2 & \cdots & 2^{c_X-1} & 2^{c_X}  \\
\vdots&\vdots&  &\vdots &\vdots \\
1& c_X & \cdots & (c_X)^{c_X-1} & (c_X)^{c_X}  \\
\end{array}\right)\ ,
\end{equation}

and

\bigskip

$\delta(u):=\prod_{i=1}^{\iota_X-1} (u+i)$ for any $u\in\mathbb{Z}_{\geq 0}$ if $\iota_X\geq 2$, $\delta(u):=1$ for any $u\in\mathbb{Z}_{\geq 0}$ if $\iota_X=1$.

\bigskip

\noindent The above discussion can be summarized as follows.

\begin{prop}\label{prop1}
Let $(X,L)$ be as in \eqref{Fano setting}. Then
\begin{equation}\label{*}
p(x,y)= \left( \sum_{i=0}^{c_X}a_i(ry-\iota_Xx)^i\right)\cdot\prod_{i=1}^{\iota_X-1} (ry-\iota_Xx+i)\ ,
\end{equation}
where $(a_0, a_1, \dots , a_{c_X})$ is the solution of \eqref{linear system Fano}.
\end{prop}

\medskip

Observe that the Hilbert curve $\Gamma_{(X,L)}$ of a
pair $(X,L)$ as in \eqref{Fano setting} is always totally reducible over
$\mathbb{C}$, i.\ e., $p(x,y)$ is the product of $n$ polynomials of degree one in $\mathbb{C}[x,y]$, because $p(x,y)= \left(
\sum_{i=0}^{c_X}a_iz^i\right)\cdot\prod_{i=1}^{\iota_X-1} (z+i)$, where
$z=ry-\iota_Xx$.

It would be interesting to know for which pairs $(X,L)$ as in \eqref{Fano setting}, $\Gamma_{(X,L)}$ is totally reducible over $\mathbb{R}$ (or $\mathbb{Q}$).
Relying on Proposition \ref{prop1} and running the Magma Program \cite{magma} (see the Appendix), with the same notation as in the database at

\medskip

\url{http://www.grdb.co.uk/forms/toricsmooth?dimension_cmp=eq&dimension=3}\ ,

\medskip

\noindent we provide a partial answer to this question. In fact, we characterize smooth Fano toric manifold of dimension $n\leq 4$ whose $p(x,y)$ is totally reducible over $\mathbb{Q}$.

\bigskip

\begin{prop}\label{prop2}
Let $(X,L)$ be as in \eqref{Fano setting} with $n\leq 4$ and assume that $X$ is toric.
Then
$p=p_{(X,L)}(x,y)$ is totally reducible over $\mathbb{Q}$
if and only if one of the cases in the two tables below occurs, where $z:=ry-\iota_Xx$ and the {\em Nos.} and $Q$ are as in the above database.

\begin{center}
\footnotesize
\begin{longtable}{ccccclcc}
\hline
{\em No.} & $n$ & $\iota_X$ & $p$ & $\left( -K_X \right)^n$ & \qquad $X=X(Q)$ & Vol$(Q)$ & $X$
\\
\hline
$3$ &
$2$ &
$1$ &
$4(z+\frac{1}{2})^2$ &
$8$ &
\begin{tabular}{l}
Vertices:   $(1,0), (0,1),$ \\
 $(-1,1), (0,-1)$ \\
Dual:   $(0,-1), (-1,-1),$ \\
 $(-1,1), (2,1)$
\end{tabular}
& $4$ & $\mathbb{F}_1$ \\
\hline
$4$ &
$2$ &
$2$ &
$-\frac{1}{2}\left(z -2\right)(z+1)$ &
$8$ &
\begin{tabular}{l}
Vertices:   $(1,0), (0,1),$ \\
$(-1,0), (0,-1)$ \\
Dual:       $(1,-1), (-1,-1),$ \\
$(-1,1), (1,1)$
\end{tabular}
& $4$ & $\mathbb{P}^1\times\mathbb{P}^1$ \\
\hline
$5$ &
$2$ &
$3$ &
$\frac{1}{2}(z+1)(z+2)$ &
$9$ &
\begin{tabular}{l}
Vertices:   $(1,0), (0,1),$ \\
$(-1,-1)$ \\
Dual:       $(2,-1), (-1,-1),$ \\
$(-1,2)$
\end{tabular}
& $3$ & $\mathbb{P}^2$ \\
\hline
\hline \smallskip
$6$ &
$3$ &
$1$ &
$\frac{25}{3}(z+\frac{1}{2})(z+\frac{2}{5})(z+\frac{3}{5})$ &
$50$ &
\begin{tabular}{l}
Vertices:   $(1,0,0), (0,1,0),$ \\
 $(0,0,1), (-1,-1,2),$ \\
 $(0,1,-1), (0,0,-1)$ \\
Dual:   $(0,-1,-1), (-1,0,-1),$ \\
$(-1,-1,-1), (2,-1,0),$ \\
$(-1,-1,0), (-1,0,1),$ \\
$(-1,4,1), (3,0,1)$
\end{tabular}
& $8$ &
\begin{tabular}{l}
the blow-up of \\
$\mathbb{P}(\mathcal{O}_{\mathbb{P}^2}\oplus\mathcal{O}_{\mathbb{P}^2}(1))$ \\
along a line \\
\end{tabular} \\
\hline
$12$ &
$3$ &
$1$ &
$\frac{25}{3}(z+\frac{1}{2})(z+\frac{2}{5})(z+\frac{3}{5})$ &
$50$ &
\begin{tabular}{l}
Vertices:   $(1,0,0), (0,1,0),$ \\
$(0,0,1), (-1,0,1),$ \\
$(0,1,-1), (0,-1,0)$ \\
Dual:   $(0,-1,-1), (-1,-1,-1),$ \\
$(-1,-1,0), (1,-1,0),$ \\
$(-1,1,2), (-1,1,-1),$ \\
$(0,1,-1), (3,1,2)$
\end{tabular}
& $8$ & $\mathbb{P}(\mathcal{O}_{\mathbb{F}_1}\oplus\mathcal{O}_{\mathbb{F}_1}(f))$ \\
\hline
$17$ &
$3$ &
$1$ &
$8(z+\frac{1}{2})^3$ &
$48$ &
\begin{tabular}{l}
Vertices:   $(1,0,0), (0,1,0),$ \\
$(0,0,1),(-1,0,1),$ \\
$(0,-1,0), (0,0,-1)$ \\
Dual:   $(0,-1,-1), (-1,-1,-1),$ \\
$(-1,1,-1), (0,1,-1),$ \\
$(-1,-1,1),(-1,1,1),$ \\
$(2,-1,1), (2,1,1)$
\end{tabular}
& $8$ & $\mathbb{F}_1\times \mathbb{P}^1$ \\
\hline
$19$ &
$3$ &
$1$ &
$9(z+\frac{1}{2})(z+\frac{1}{3})(z+\frac{2}{3})$ &
$54$ &
\begin{tabular}{l}
Vertices:   $(1,0,0), (0,1,0),$ \\
$(0,0,1), (-1,0,1),$ \\
$(0,-1,-1)$ \\
Dual:   $(0,-1,-1), (-1,-1,-1),$ \\
$(-1,2,-1), (-1,-1,2),$ \\
$(0,2,-1), (3,-1,2)$
\end{tabular}
& $6$ & $\mathbb{P}(\mathcal{O}_{\mathbb{P}^1}^{\oplus 2}\oplus\mathcal{O}_{\mathbb{P}^1}(1))$\\
\hline
$22$ &
$3$ &
$1$ &
$9(z+\frac{1}{2})(z+\frac{1}{3})(z+\frac{2}{3})$ &
$54$ &
\begin{tabular}{l}
Vertices:   $(1,0,0), (0,1,0),$ \\
$(0,0,1),(-1,0,0),$ \\
$(0,-1,-1)$ \\
Dual:   $(1,-1,-1), (-1,-1,-1),$ \\
$(1,2,-1), (-1,2,-1),$ \\
$(-1,-1,2), (1,-1,2)$
\end{tabular}
& $6$ &
\begin{tabular}{l}
$\mathbb{P}^2\times \mathbb{P}^1$ \\
\end{tabular} \\
\hline
$23$ &
$3$ &
$4$ &
$\frac{1}{6}(z+1)(z+2)(z+3)$ &
$64$ &
\begin{tabular}{l}
Vertices:   $(1,0,0), (0,1,0),$ \\
$(0,0,1), (-1,-1,-1)$ \\
Dual:   $(3,-1,-1), (-1,3,-1),$ \\
$(-1,-1,3), (-1,-1,-1)$
\end{tabular}
& $4$ & $\mathbb{P}^3$ \\
\hline
\\
\caption{Toric Fano for $n=2, 3$ with $p$ totally reducible over $\mathbb{Q}$.}
\label{types n=2,3}
\end{longtable}
\end{center}

\begin{center}
\footnotesize
\begin{longtable}{ll}
\hline
\\
{\em Nos.} $=\ 37, 43, 59, 68, 97, 105, 106, 112, 114, 115, 130, 133, 135, 136, 138, 143, 147.$
\\
\hline
\\
\caption{Toric Fano for $n=4$ with $p$ totally reducible over $\mathbb{Q}$.}
\label{types n=4}
\end{longtable}
\end{center}
\end{prop}

\begin{rem}
{\em Results in line with Proposition \ref{prop2} can be obtained by using the same Magma Program also for $n=5,6$, but the lists became very long.}
\end{rem}

\medskip

\begin{lem}\label{Fano}
Let $(X,L)$ be a polarized manifold of dimension $n\geq 2$ and let $r,m$ be two positive integers with $\mathrm{gcd}(r,m)=1$.
If $\mathrm{rk}\langle K_X,L\rangle=1$ and $p_0(r,m,0)=0$, then $X$ is Fano of index $\iota_X = k m$ and
$L = k r H$ for some positive integer $k$, where $H$ is the fundamental divisor.
\end{lem}

\begin{proof}
Since $\mathrm{rk}\langle K_X,L\rangle=1$, we have
$K_X+aL=\mathcal{O}_X$ for some $a\in\mathbb{Q}$. Moreover,
\eqref{atinfty} gives
$$0=n!\ p_0(r,m,0) = (rK_X+mL)^n=(m-ra)^nL^n,$$
hence $a=\frac{m}{r}>0$.
Therefore $X$ is a Fano manifold such
that $rK_X+mL=\mathcal{O}_X$. Let $\iota_X$ be the index of $X$
so that $-K_X=\iota_XH$ for some ample $H\in\mathrm{Pic}(X)$. Note
that Pic$(X)$ is torsion free. Moreover, we have $mL=r(-K_X)=r \iota_XH=A$ for some ample line bundle $A$ on $X$. Write $m=\sigma m'$
and $r \iota_X=\sigma s'$ for some positive integers $m',s'$, where
$\sigma:=\mathrm{gcd}(m,r \iota_X)$. Then $A$ is divisible by $m$ and $r \iota_X$ in
Pic$(X)$ and this implies that $A=m's'\sigma M$ for some ample
line bundle $M$ on $X$, because $\mathrm{gcd}(m',r)=\mathrm{gcd}(m,s')=1$. Thus we get
$L=s'M, H=m'M$ and then $-K_X=\iota_X m'M$. Since $\iota_X$ is the index
of $X$, we conclude that $m'=1$. Hence $M=H$, $L=s'H$,
$m=\sigma=\mathrm{gcd}(m,r \iota_X)$. As $\mathrm{gcd}(m,r)=1$, we deduce that $m$ divides
$\iota_X$, that is, $\iota_X=km$ for some positive integer $k$. So we
get
$$\mathcal{O}_X=rK_X+mL=r(-\iota_X H)+ms'H=r(-km H)+ms'H=m(s'-kr)H,$$
i.\ e., $s'=kr$, hence $L=krH$.
\end{proof}

For Fano manifolds of sufficiently large index we immediately get the following characterization.

\begin{cor}\label{Fano2} Let $(X,L)$ be as in Lemma $\ref{Fano}$
and suppose that $m>\frac{n+1}{2}$. Then
$X$ is a Fano manifold of index $\iota_X=m$ and $L := r\left(\frac{-K_X}{m}\right)$
if and only if $\mathrm{rk}\langle K_X,L\rangle=1$ and
$$p_{(X,L)}(x,y):= \left(a_{c_X}(ry-mx)^{c_X} + \dots + a_1(ry-mx) + \frac{1}{(m-1)!}\right)\cdot\prod_{i=1}^{m-1} (ry-mx+i)\ ;$$

\noindent moreover, in this situation, the coefficients $a_i$'s are given by \eqref{linear system Fano}.
\end{cor}

\begin{proof}
The ``only if" part follows easily from Proposition \ref{prop1}. To prove the converse, by applying
Lemma \ref{Fano} we know that $X$ is Fano of index $\iota_X = km$ and $L=krH$ for some positive integer $k$.
Combining the assumption with the upper bound for $\iota_X$ we get
$k \frac{n+1}{2} < km \leq n+1$. Thus $k=1$, which implies the assertion.
\end{proof}

\begin{rem}\label{rem1}
{\em Given a Fano manifold $X$ of dimension $n$ and index $\iota_X\geq n-2$, it is known that there exists a smooth element
$Y \in |H|$. This is obvious for $\iota_X=n+1$ and $n$; it follows from Fujita's theory of del Pezzo manifolds \cite[$\S 8$]{Fu} for
$\iota_X=n-1$ and from \cite{Me} for $\iota_X=n-2$. Note that $-K_Y=(\iota_X-1)H_Y$ by adjunction. In particular, if $n\geq 3$ and $(X,H)$ is
a del Pezzo manifold, then $(Y,H_Y)$ is also a del Pezzo manifold, and similarly, if $n\geq 4$ and $(X,H)$ is
a Mukai manifold, then $(Y,H_Y)$ is a Mukai manifold too. A consequence of this fact is that for $\iota_X\geq n-2$ we can always apply
an inductive argument up to the surface case to compute $h^0(tH)$ for $t=1,\dots,c_X\leq 3$.}
\end{rem}

\medskip

In particular, we get the following explicit characterization of Fano manifolds of index $\iota_X\geq \dim X-1$
in terms of their HC. For the case $\iota_X=\dim X-2$, see Proposition \ref{Mukai}.

\begin{thm}\label{projective space,quadric,del Pezzo}
Let $(X,L)$ be a polarized manifold of dimension $n\geq 2$
and let $r$ be a positive integer.
\begin{enumerate}

\item[(i)] Suppose $\mathrm{gcd}(r,n+1)=1$. Then $(X,L)=\big(\mathbb{P}^n,\mathcal{O}_{\mathbb{P}^n}(r)\big)$ if and only if
$\mathrm{rk}\langle K_X,L\rangle=1$ and
$$p_{(X,L)}(x,y) = \frac{1}{n!}\prod_{i=1}^{n}\left(ry-(n+1)x+i\right).$$

\item[(ii)] Suppose $\mathrm{gcd}(r,n)=1$. Then $(X,L)=\big(\mathbb{Q}^n,\mathcal{O}_{\mathbb{Q}^n}(r)\big)$ if and only if
$\mathrm{rk}\langle K_X,L\rangle=1$ and
$$p_{(X,L)}(x,y) = \left(\frac{2}{n!} \left(ry-nx\right)+\frac{1}{(n-1)!}\right) \prod_{i=1}^{n-1}(ry-nx+i).$$

\item[(iii)] Suppose $\mathrm{gcd}(r,n-1)=1$. Then $X$ is a Fano manifold of index $\iota_X = n-1$ and $L := \frac{r}{n-1}\left(-K_X\right)$
if and only if $\mathrm{rk}\langle K_X,L\rangle=1$ and
{\small{
$$p_{(X,L)}(x,y) = \left(\frac{d}{n!}\left(ry-(n-1)x\right)^2 + \frac{(n-1)d}{n!}\left(ry-(n-1)x\right) + \frac{1}{(n-2)!}\right) \prod_{i=1}^{n-2}(ry-(n-1)x+i),$$ }}
where $d = \left(\frac{1}{n-1}\left(-K_X\right)\right)^n$, and $d\not=8$ if $n=3$, $d\not=9$ if $n=2$.
\end{enumerate}
\end{thm}
\noindent
Notice that condition $\iota_X = n-1$ implies that $(X,\frac{1}{\iota_X}(-K_X))$ is a del Pezzo manifold (while the converse is not true).
In particular, $d \leq 7$ if $n = 3$ and $d \leq 8$ if $n=2$.
For the classification of del Pezzo manifolds see \cite [(8.11)]{Fu}. For more details we refer
to the comments before Table \ref{the geography of HC}.

\begin{proof} The ``only if'' part follows from Proposition \ref{prop1}. We have to determine the coefficients in
\eqref{linear system Fano}. The computation of $a_0$ in case (i) and of
$(a_0, a_1)$ in case (ii) is immediate. In case (iii) we have to compute $(a_0,a_1,a_2)$. The matrix $U$ in \eqref{linear system Fano}
is
$$
U=\left( \begin{array}{ c c c }
1 & 0 & 0 \\
1 & 1 & 1 \\
1 & 2 & 4 \\
\end{array}\right)\ .
$$
To determine the column on the right hand of \eqref{linear system Fano} set $H=\frac{1}{n-1}(-K_X)$; then
$(X,H)$ is a del Pezzo manifold, hence
for any smooth element $Y \in |H|$ the pair $(Y,H_Y)$ is a del Pezzo manifold too, by Remark \ref{rem1}. Then,
with the help of the exact sequences
$$0 \to \mathcal O_X \to H \to H_Y \to 0 \qquad \text{\rm{and}} \qquad
0 \to H \to 2H \to 2H_Y \to 0,$$
we get, by induction,
\begin{equation*}\label{h0}
h^0(H) = n-1+d \qquad \text{\rm{and}} \qquad h^0(2H)=\binom{n}{2}+(n+1)d \ ,
\end{equation*}
where $d=H^n$. Thus the column on the right hand of \eqref{linear system Fano} is the transpose of the vector
$$\Big(\frac{1}{(n-2)!}, \frac{n-1+d}{(n-1)!}, \frac{n(n-1)+2(n+1)d}{n!}\Big),$$
and therefore, solving \eqref{linear system Fano} we get
$$(a_0, a_1, a_2) = \Big(\frac{1}{(n-2)!}, \frac{(n-1)d}{n!}, \frac{d}{n!} \Big).$$

The ``if part'' follows from Corollary \ref{Fano2} in cases (i) and (ii), letting $m=n+1$
and $m=n$ respectively. In case (iii) letting $m=n-1$, Corollary \ref{Fano2} applies again when
$n \geq 4$. So, let $n=3$, recalling Lemma \ref{Fano} we see that $4 = n+1 \geq \iota_X = k(n-1)$.
Hence $k \leq 2$. If $k=1$ we are done; on the other hand if $k=2$, then $\iota_X=4$, hence $X = \mathbb P^3$
by the Kobayashi--Ochiai theorem,
but in this case $d=\big(\frac{-K_X}{2}\big)^3 = 8$, a contradiction.
A similar discussion can be done for $n=2$: see also the example at
the end of Section \ref{backgr}.
\end{proof}

\bigskip

Let us comment here on the ``geography'' of $\Gamma=\Gamma_{(X,L)}$ in cases (i)--(iii) of Theorem \ref{projective space,quadric,del Pezzo}.
For simplicity, we will use coordinates $(u,v)$, given by $u:=x-\frac{1}{2}$ and $v:=y$, instead of $(x,y)$.

\medskip

In case (i), $\Gamma=\ell_1 + \dots + \ell_{n}$ consists
of $n$ parallel evenly spaced lines
$$\ell_i:\ rv-(n+1)u+\left(i-\frac{n+1}{2}\right)=0$$
with slope $\frac{n+1}{r}$, for $i=1,...,n$. This holds also for $\Gamma_{\mathbb R}$ and $\Gamma_{\mathbb Q}$.

\medskip

In case (ii), $\Gamma=\ell_0 + \ell_1 + \dots + \ell_{n-1}$ consists
of $n$ parallel lines with slope $\frac{n}{r}$. The lines $\ell_i$ for $i=1, \
\dots n-1$ have equations
$$\ell_i:\ rv-nu+\left(i-\frac{n}{2}\right)=0\ ,$$
hence they are evenly spaced. On the other hand, $\ell_0$ has equation
$rv-nu=0 .$
Clearly it may happen that $\ell_0$ overlaps one of the $\ell_i$'s. This happens if and only if
$n=2m$ and $i=m$. So, for $n=2m-1$ odd, there is no overlapping and $\ell_0$ is the bisecant of the strip between
$\ell_{m-1}$ and $\ell_m$. On the contrary, for $n=2m$, $\Gamma$ is non-reduced, having
the line $\ell_m$ as component of multiplicity $2$. This discussion applies also to $\Gamma_{\mathbb R}$ and $\Gamma_{\mathbb Q}$.

\medskip

Now, consider case (iii) and let $\ell_1, \dots \ell_{n-2}$ be the lines
$$\ell_i:\ rv-(n-1)u+\left(i-\frac{n-1}{2}\right)=0$$
defined by the
linear factors in the expression of $p$.
All of them are parallel each other with slope $\frac{n-1}{r}$ and evenly spaced.
Let $G$ be the conic defined by the residual polynomial.
Then
$$\Gamma =\ell_1 + \dots + \ell_{n-2} + G$$
and, up to the multiplicative factor $\frac{d}{n!}$, the equation of $G$ is
\[ [u\ v\ 1]\
\begin{bmatrix}
(n-1)^2 & -r(n-1) & 0\\
-r(n-1) & r^2 &  0\\
0 & 0 & -h\\ \end{bmatrix}
\begin{bmatrix} u\\ v\\ 1\\ \end{bmatrix} = 0,
\]
where
$$h: =  \frac{(n-1)}{4d}\big[(n-1)d-4n \big].$$
From the complex point of view, the conic $G$ consists of two
distinct or coinciding lines $\lambda$ and $\lambda'$, parallel to the $\ell_i$'s and, possibly, partially coinciding
with some of them. Note that
$\ell_i \subset \mathbb A^2_{\mathbb Q}$ for any $i=1, \dots, n-2$, hence the latter possibility requires that
the term $h$ is the square of a rational number.
Before to see this, let us look at $G$ from the real point of view (denoted by $G_{\mathbb{R}}$).
Here we assume $n\geq 3$, since the case of surfaces has already been discussed in the example at the end of Section \ref{backgr}.
We have $h = k^2$ with $k\in \mathbb{R}$ if and only if
$$
d \geq \frac{4n}{n-1}.
$$
In particular, this implies $d\geq 5$ and $d\geq 6$ if $n=3$ or $4$. Look at the del Pezzo manifold $(X,H)$.
Recalling Fujita's classification \cite[(8.11)]{Fu}, consider that $d \leq 4$ if $n \geq 7$.
Therefore, for $n \geq 7$ we get $h<0$, that is, $G_{\mathbb R} = \emptyset$ (and then $G_{\mathbb Q} = \emptyset$ a fortiori).
In this case $\Gamma_{\mathbb R}$
as well as $\Gamma_{\mathbb Q}$, simply consist of $n-2$ evenly spaced parallel lines.
Consider also that
$d\leq 8$ for $n=3$, $d\leq 6$ for $n=4$, and $d\leq 5$ if either $n=5$, or $n=6$. However, case $d = 8$ in which
$(X,H) = (\mathbb P^3, \mathcal O_{\mathbb P^3}(2))$ was excluded from (iii); in fact, it fits into (i).

\smallskip

Thus, a case-by-case analysis leads to
the following further conclusions concerning $G_{\mathbb R}$, when not empty.

\smallskip

For $n=3$, we have the following two possibilities:
\begin{enumerate}
\item[a)] $G_{\mathbb R} = \lambda + \lambda'$ is the union of two distinct lines, both distinct from $\ell_1$, and this happens for
$d = 7$, in which case $X=\mathbb P(\mathcal O_{\mathbb P^2}(2) \oplus \mathcal O_{\mathbb P^2}(1))$ and $H$ is the tautological
line bundle;

\smallskip
\item[b)] $G_{\mathbb R}=2\lambda$ is a double line if $d = 6$; in this case $\lambda= \ell_1$, hence
$\Gamma_{\mathbb R} = 3\ell_1$ is a triple line; here, either $(X,H) = (\mathbb P^1 \times \mathbb P^1 \times \mathbb P^1, \mathcal O(1,1,1))$, or
$X=\mathbb{P}(T_{\mathbb{P}^2})$ and $H$ is the tautological line bundle.
\end{enumerate}

If $n=4$, then $G_{\mathbb R} = \lambda + \lambda'$ is the union of two distinct lines, $\lambda = \ell_1$, $\lambda' = \ell_2$, so that
$\Gamma_{\mathbb R} = 2(\ell_1 + \ell_2)$. This happens for $d=6$ and it corresponds
to $(X,H) = (\mathbb P^2 \times \mathbb P^2, \mathcal O(1,1))$.

\smallskip
If $n=5$, then $G_{\mathbb R} = 2\lambda$ with $\lambda = \ell_2$, so that
$\Gamma_{\mathbb R} = \ell_1 + 3\ell_2 + \ell_3$. In this case, $d=5$, and the corresponding $(X,H)$
is the hyperplane section of the Grassmannian $\mathbb{G}(1,4)$ embedded in $\mathbb P^9$ via the Pl\"ucker embedding.

\smallskip
Finally, if $n=6$ then $G_{\mathbb R} = \lambda + \lambda'$ is the union of two distinct lines, where $\lambda = \ell_2$ and $\lambda' = \ell_3$. So
$\Gamma_{\mathbb R} = \ell_1 + 2(\ell_2 + \ell_3) + \ell_4$.
In this case, we have $d=5$ and $X$ is the Grassmannian $\mathbb{G}(1,4)$ embedded by $H$ in $\mathbb P^9$ via the Pl\"ucker embedding.

\medskip

As to the situation for $G_{\mathbb Q}$ (when $G_{\mathbb R} \not= \emptyset$), we note the following fact.
First of all, we get $h = 0$ when $(n,d) =
(3,6), (5,5)$. Moreover, $h=k^2$ for some $k \in \mathbb Q$ when: $(n,d) =
(3,8), (4,6)$ and $(6,5)$ (in which cases $h = (1/2)^2$). On the other hand,
$h=k^2$ with $k \not\in \mathbb Q$ if and only if $(n,d)=(3,7)$ (here $h = 1/7$).
Therefore, the description of $\Gamma_{\mathbb Q}$ is the same as that
given for $\Gamma_{\mathbb R}$, up to regarding $\lambda, \lambda'$ and the $\ell_i$'s as lines in $\mathbb A^2_{\mathbb Q}$,
except when $(n, d)=(3, 7)$, in which case $G_{\mathbb Q}=\emptyset$, so that $\Gamma_{\mathbb Q} = \ell_1$, the
line of equation $2u-rv=0$.

\medskip

The following table summarizes the above discussion about the Hilbert curves in cases (i)--(iii) of Theorem \ref{projective space,quadric,del Pezzo}.

\bigskip

\begin{center}
\footnotesize
\begin{longtable}{cll}
\hline
$n$ & $(X,L)$ & $\Gamma:=\Gamma_{(X,L)}$ in $(u,v)$ coordinates
\\
\hline
\\
$\geq 2$ &
\begin{tabular}{l}
$(\mathbb{P}^n,\mathcal{O}_{\mathbb{P}^n}(r))$, \\
for $r\geq 1$ and $\mathrm{gcd}(r,n+1)=1$
\end{tabular} &
\begin{tabular}{l}
$\Gamma = \ell_1+ \dots + \ell_n$ \\
where $\ell_i:\ v=\frac{(n+1)}{r}u+\frac{1}{r}\left(\frac{n+1}{2}-i\right),$\\
for $i=1,...,n$
\end{tabular} \\
\\
\hline
\hline
\\
$2m-1\geq 3$ &
\begin{tabular}{l}
$(\mathbb{Q}^{2m-1},\mathcal{O}_{\mathbb{Q}^{2m-1}}(r))$, \\
for $r\geq 1$ and $\mathrm{gcd}(r,2m-1)=1$
\end{tabular} &
\begin{tabular}{l}
$\Gamma = \ell_0+\ell_1+ \dots + \ell_{2m-2}$\\
where $\ell_0:\ v=\frac{2m-1}{r}u$ and \\
$\ell_i:\ v=\frac{2m-1}{r}u+\frac{1}{r}\left(\frac{2m-1-2i}{2}\right)$\\
for $i=1,...,2m-2$
\end{tabular} \\
\\
$2m\geq 2$ &
\begin{tabular}{l}
$(\mathbb{Q}^{2m},\mathcal{O}_{\mathbb{Q}^{2m}}(r))$, \\
for $r\geq 1$ and $\mathrm{gcd}(r,2m)=1$
\end{tabular} &
\begin{tabular}{l}
$\Gamma = \ell_0+\ell_1+ \dots + \ell_{2m-2}$ \\
where $\ell_0:\ v=\frac{2m}{r}u$ and \\
$\ell_i:\ v=\frac{2m}{r}u+\frac{1}{r}\left(m-i\right)$\\
for $i=1,...,2m-1$
\end{tabular} \\
\\
\hline
\hline
\\
$\geq 3$ &
\begin{tabular}{l}
$X$ Fano \\
of index $n-1$ and \\
$L=rH \left(=\frac{r}{n-1}\left(-K_X\right)\right)$, \\
for $r\geq 1$ and $\mathrm{gcd}(r,n-1)=1$
\end{tabular} &
\begin{tabular}{l}
$\Gamma = \ell_1+ \dots + \ell_{n-2}+G$ \\
where $\ell_i:\ v=\frac{n-1}{r}u+\frac{1}{r}\left(\frac{n-1}{2}-i\right)$\\
for $i=1,...,2m-1$ and \\
$G:\ \left[(n-1)u-rv\right]^2-h=0$ \\
$h=\frac{n-1}{4d}\left[(n-1)d-4n\right], d=H^n$
\end{tabular} \\
& & \\
& Further information on $G_{\mathbb{R}}$ and $G_{\mathbb{Q}}$: &  \\
& & \\
$\geq 7$ &  &
\begin{tabular}{l}
$G_{\mathbb{R}}=G_{\mathbb{Q}}=\emptyset;\ h<0$, $d\leq 4$
\end{tabular} \\
& & \\
$6$ &
\begin{tabular}{l}
$X=\mathbb{G}(1,4)\subset\mathbb{P}^9$
\end{tabular}
&
\begin{tabular}{l}
$G_{\mathbb{R}}=G_{\mathbb{Q}}=\ell_2+\ell_3;\ h=\left(\frac{1}{2}\right)^2$, $d=5$
\end{tabular} \\
& & \\
$5$ &
\begin{tabular}{l}
$X=\mathbb{G}(1,4)\cap\mathbb{P}^8$
\end{tabular}
&
\begin{tabular}{l}
$G_{\mathbb{R}}=G_{\mathbb{Q}}=2\ell_2;\ h=0$, $d=5$
\end{tabular} \\
& & \\
$4$ &
\begin{tabular}{l}
$X=\mathbb{P}^2\times\mathbb{P}^2$
\end{tabular}
&
\begin{tabular}{l}
$G_{\mathbb{R}}=G_{\mathbb{Q}}=\ell_1+\ell_2;\ h=\left(\frac{1}{2}\right)^2$, $d=6$
\end{tabular}\\
& & \\
$3$ &
\begin{tabular}{l}
$X=\mathbb{P}(\mathcal{O}_{\mathbb{P}^2}(2)\oplus\mathcal{O}_{\mathbb{P}^2}(1))$ \\
\end{tabular}
&
\begin{tabular}{l}
$G_{\mathbb{R}}=\lambda+\lambda', \lambda\neq\lambda'$ and both $\not= \ell_1$ \\
$G_{\mathbb{Q}}=\emptyset;\ h=\frac{1}{7}$, $d=7$
\end{tabular} \\
& & \\
$3$ &
\begin{tabular}{l}
$X=\mathbb{P}^1\times\mathbb{P}^1\times\mathbb{P}^1$, or $\mathbb{P}(T_{\mathbb{P}^2})$
\end{tabular}
&
\begin{tabular}{l}
$G_{\mathbb{R}}=G_{\mathbb{Q}}=2\ell_1;\ h=0$, $d=6$
\end{tabular} \\
\\
\hline
\hline
\\
\caption{HC of pairs $(X,L)$ as in \eqref{Fano setting}, with $\iota_X\geq n-1$.}
\label{the geography of HC}
\end{longtable}
\end{center}

\begin{rem}\label{analog-conjecture}
{\em Consider the following polarized threefolds: $(X,L)=\big(\mathbb P^3, \mathcal O_{\mathbb P^3}(2)\big)$,
$(X',L')=\big(\mathbb Q^3, \mathcal O_{\mathbb Q^3}(1)\big)$, and the
del Pezzo threefold $(X'',L'')$ of degree $7$. According to Theorem \ref{projective space,quadric,del Pezzo},
their Hilbert curves, $\Gamma, \Gamma', \Gamma''$ respectively, have the following canonical equations
in terms of coordinates $u=x-\frac{1}{2}, v=y$ :
$$p(\frac{1}{2}+u,v)= \ \frac{1}{6}(v-2u+\frac{1}{2})(v-2u)(v-2u-\frac{1}{2}) = 0,$$
$$p'(\frac{1}{2}+u,v)= \ \frac{1}{3}(v-3u) (v-3u-\frac{1}{2})(v-3u+\frac{1}{2}) = 0,$$
$$p''(\frac{1}{2}+u,v)= \ \frac{7}{6}(v-2u) \big((v-2u)^2-\frac{1}{7}\big) = 0.$$
Look at them from the real point of view: $\Gamma_{\mathbb R}$ consists of three parallel lines
(symmetric with respect to the origin), with slope $2$, evenly spaced, with step $\frac{1}{2}$ on the $v$-axis.
The same is true for $\Gamma'_{\mathbb R}$ except for the slope, which is $3$, and also for
$\Gamma''_{\mathbb R}$, in which case the slope is $2$ again, but here the step on the $v$-axis is $\frac{1}{\sqrt{7}}$,
an irrational number. Clearly, the three curves are equivalent each other from the real affine point of view.
Moreover $\Gamma$ and $\Gamma''$ are similar from the Euclidian point of view. However, they are different in terms
of their ``geography'' (either different slopes, or different steps on the $v$-axis). Moreover, the
difference between $\Gamma$ and $\Gamma''$ becomes even more evident if we consider their arithmetic,
looking at ${\Gamma}_{\mathbb Q}$ and $\Gamma''_{\mathbb Q}$:
the former consists of three lines, while the latter consist of the single line $v-2u=0$, since the factor $(2u-v)^2-\frac{1}{7}$
is irreducible over $\mathbb Q$.}
\end{rem}

The facts pointed out in Remark \ref{analog-conjecture} should be taken into account in formulating a conjecture characterizing, e.\ g., the projective space, similar but a posteriori much easier than \cite[Conjecture C$(n,r)$]{L1}, as follows. Let $(X,L)$ be a polarized manifold of dimension $n\geq 2$ with $\text{\rm{rk}}\langle K_X,L\rangle = 1$, and let $r$ be a positive integer such that $\mathrm{gcd}(r,n+1)=1$. Then $(X,L)=(\mathbb P^n, \mathcal O_{\mathbb P^n}(r)\big)$ if and only if the Hilbert curve $\Gamma$ of $(X,L)$ consists of $n$ distinct lines (symmetric with respect to the origin), parallel each other with slope $\frac{n+1}{r}$ and evenly spaced. This conjecture is true in view of
Proposition \ref{projective space,quadric,del Pezzo} (i). Moreover, a consequence of the next result is that this conjecture
is still true provided that $rK_X+(n+1)L$ is nef regardless of the assumption $\text{\rm{rk}}\langle K_X,L\rangle = 1$.
This change of perspective will be the starting point for the next section and it will allow us to prove Conjecture $C(n,r)$ in \cite{L1} (cf. Theorem \ref{thm2}) under an extra assumption.

\begin{thm}\label{thm1}
Let $(X,L)$ be a smooth polarized manifold of dimension $n\geq 2$ and
let $r$ be a positive integer such that $\mathrm{gcd}(r,n+1)=1$. Then $(X,L)=(\mathbb{P}^n,\mathcal{O}_{\mathbb{P}^n}(r))$
if and only if
$rK_X+(n+1)L$ is nef and $p_{(X,L)}(x,y)=\frac{1}{n!}\prod_{i=1}^{n}(ry-(n+1)x+i)$.
\end{thm}

\begin{proof}
In view of Theorem \ref{projective space,quadric,del Pezzo} (i) we only need to prove the ``if part'' under the assumption that
\begin{equation}\label{rango2}
\mathrm{rk}\langle K_X,L\rangle=2.
\end{equation}
First let $n=2$. By comparing the expression of $p_{(X,L)}(x,y)$ in the statement with that holding for any polarized surface $(X,L)$ (cf.\ \eqref{surf}),
we see that $K_X^2=9$, $K_X \cdot L = -3r$ and $\chi(\mathcal O_X)=1$. The last two conditions imply that $X$ is a rational surface and then the first condition says that
$X = \mathbb P^2$, which contradicts \eqref{rango2}. Let $n \geq 3$. The expression of $p_{(X,L)}(x,y)$ shows that the point $(r:n+1:0)$ belongs to the
projective closure of the HC, $\overline{\Gamma_{(X,L)}} \subset \mathbb P^2$, hence
$$0 = n!\ p_0(r,n+1,0) = \left(rK_X+(n+1)L\right)^n$$
by \eqref{atinfty}. Therefore $rK_X+(n+1)L$ is nef but not big. By Remark \ref{morphism} we know that there exists a morphism $\varphi:X \to Y$
with $\dim Y<\dim X$ such that $rK_X+(n+1)L=\varphi^*D$ for some nef line bundle $D$ on $Y$. Then $rK_F+(n+1)L_F=\mathcal{O}_F$ by adjunction,
where $F$ is a general fiber of $\varphi$.
Thus $-K_F = \frac{n+1}{r}L_F$ and since $L_F$ is ample we conclude that $F$ is a
Fano manifold. Moreover, the assumption $\mathrm{gcd}(r,n+1)=1$ implies that
$-K_F \cdot \gamma = (n+1) \frac{L_F \cdot \gamma}{r} \geq n+1$ for every rational
curve $\gamma\subset F$. Then the index $i_F$ of $F$ satisfies
$\dim X+1 \geq \dim F+1 \geq i_F \geq n+1,$ i.\ e., $\dim F=\dim X$. So
$Y$ is a point and $X=F=\mathbb{P}^n$, which contradicts \eqref{rango2} again.
\end{proof}

In line with Theorem \ref{projective space,quadric,del Pezzo}, we can also obtain a characterization of
pairs $(X,L)$ as in \eqref{Fano setting} with $\iota_X=n-2$ and $\mathrm{rk}\langle K_X,L\rangle=1$ in terms of their HC, provided that $n\geq 6$.
Actually, under this assumption, we can rely on Corollary \ref{Fano2} again and we just need to determine the coefficients $a_i$'s for $i=0,\dots,3$.
We can do that by the same procedure as in Theorem \ref{projective space,quadric,del Pezzo}, relying on Remark \ref{rem1}. The final output is the following result.

\begin{prop}\label{Mukai}
Let $(X,L)$ be a polarized manifold of dimension $n\geq 6$
and let $r$ be a positive integer such that $\mathrm{gcd}(r,n-2)=1$. Assume
that $\mathrm{rk}\langle K_X,L\rangle=1$. Then
$(X,L)$ is as in \eqref{Fano setting} with $\iota_X=n-2$ if and only if
$p_{(X,L)}(x,y)$ is as in Corollary $\ref{Fano2}$ with
$$a_0=\frac{1}{(n-3)!},\ a_1=\frac{1}{n!}\left[\left(\frac{d}{2}+1\right)n^2-(2d+1)n+2d\right],\ a_2=\frac{3d}{2\ n!}(n-2),\ a_3=\frac{d}{n!}\ ,$$
where $d:= \left(\frac{-K_X}{n-2}\right)^n$.
\end{prop}
In fact, this 4-tuple $(a_0,a_1,a_2,a_3)$ is the solution of \eqref{linear system Fano}, because the column on the right hand of \eqref{linear system Fano}
is the transpose of the vector
$$\left( \begin{array}{ c c c c }
\frac{1}{(n-3)!} , & \frac{1}{(n-2)!}\left[ n-1+\frac{d}{2}\right] , & \frac{2}{(n-1)!}\left[ \binom{n}{2}+(n+2)\frac{d}{2}\right] , & \frac{6}{n!}\left[\binom{n+1}{3}+\left( \frac{(n+1)(n+4)}{2}\right)\frac{d}{2}\right]
\end{array}\right)\ .$$

\noindent We finally note that, in principle, Algorithm $1$ in the Appendix allows one to compute $p_{(X,L)}(x,y)$ for any pair as in \eqref{Fano setting},
provided that $h^0(tH)$ is known for every $t=1,\dots,c_X$.

\medskip

\section{Case $\mathrm{rk}\langle K_X,L\rangle=2$: Fano fibrations over curves}\label{subsection1}

Here and in the next section we are concerned with Fano fibrations of coindex $\leq 1$. We implicitly assume that $\mathrm{rk}\langle K_X,L\rangle=2$.

First of all, we describe a method for obtaining the canonical equation of the HC of a Fano fibration over a smooth curve. Let $(X,L)$ be a Fano fibration over a smooth irreducible curve $C$ via a morphism $\varphi:X \to C$, let $F$ be a general fiber of $\varphi$, and suppose that $rK_X+\iota_FL=\varphi^*\mathcal A$ for some non-trivial line bundle $\mathcal A$ on $C$.
Thus $rK_X+\iota_FL\equiv tF$, where $t=\deg\mathcal A \neq 0$.
Consider the following exact sequences defined by the restriction to $F$:
$$0\to xK_X+yL - F \to xK_X+yL\to xK_F+yL_F\to 0\ ,$$
$$0\to xK_X+yL - 2F \to xK_X+yL-F\to xK_F+yL_F\to 0\ ,$$
$$\cdots$$
$$0\to xK_X+yL - |t|F \to xK_X+yL - (|t|-1)F\to xK_F+yL_F\to 0\ .$$
Let $\mathrm{sgn}(t)$ be the sign of $t$, so that $|t|=\mathrm{sgn}(t)t$.
Then, due to the additivity of the Euler--Poincar\'e characteristic $\chi$ for exact sequences, by recursion we get
$$p_{(X,L)}(x,y)=p_{(X,L)}(x-\mathrm{sgn}(t)r,y-\mathrm{sgn}(t)\iota_F)+\mathrm{sgn}(t)t\cdot p_{(F,L_F)}(x,y)\ .$$

\medskip

\noindent Since $F$ is Fano with
$rK_F+\iota_FL_F=\mathcal{O}_F$, we know that
$$p_{(F,L_F)}(x,y)=R_F(x,y)\cdot\prod_{j=1}^{\iota_F-1}\big(ry-\iota_Fx+j\big)\ ,$$
where $R_F(x,y)$ is the first factor of \eqref{*} in Proposition \ref{prop1}.
Thus, we finally obtain that
\begin{equation}\label{R_F1}
R(x,y)=R(x-\mathrm{sgn}(t)r,y-\mathrm{sgn}(t)\iota_F)+\mathrm{sgn}(t)t\cdot R_F(x,y),
\end{equation}
where
$p_{(X,L)}(x,y)=R(x,y)\cdot\prod_{j=1}^{\iota_F-1}\big(ry-\iota_Fx+j\big)$.
So,
(\ref{R_F1}) with further suitable conditions allows us to determine $R(x,y)$, once we know $R_F(x,y)$ (cf.\ Algorithm \ref{Alg2} in the Appendix).

\medskip

The following example shows how this method works for Fano fibrations of coindex $0$ over a smooth curve (see also \cite[Proposition 2.1]{L1}).

\begin{ex}\label{ex scrolls over curve}
{\em
Let $X=\mathbb{P}(\mathcal{E})$ for some vector bundle $\mathcal{E}$ of rank $n$ over a smooth curve $C$ of genus $g\geq 0$. Let $\xi$ be the tautological line bundle on $X$ and consider an ample line bundle $L$ numerically equivalent to $r\xi+bF$ for some integers $b$, and $r>0$, where $F\cong\mathbb{P}^{n-1}$ is a fiber of the bundle projection $\varphi:X\to C$. Since in this case $\iota_F=n$, we have $$p_{(X,L)}(x,y)=R(x,y)\cdot\prod_{j=1}^{n-1}\big(ry-nx+j\big)$$ with $R(x,y)=\alpha x+\beta y+\gamma$ for some $\alpha,\beta,\gamma\in\mathbb{Q}$.
Here $rK_X+nL\equiv tF$, where $t:=(2g-2+e)r+nb\neq 0$ with $e:=\deg\mathcal{E}$. It thus follows from \eqref{R_F1} that
$$\mathrm{sgn}(t)(\alpha r+\beta n)=R(x,y)-R(x-\mathrm{sgn}(t)r,y-\mathrm{sgn}(t)n)=\mathrm{sgn}(t)t\cdot R_F(x,y)=\frac{\mathrm{sgn}(t)t}{(n-1)!}\ ,$$
by Theorem \ref{projective space,quadric,del Pezzo} (i).
Moreover, since $q:=h^1(\mathcal{O}_X)=g$,
$$1-g=\chi(\mathcal{O}_X)=p_{(X,L)}(0,0)=R(0,0)\cdot (n-1)!= \gamma (n-1)!\ .$$
Furthermore, up to numerical equivalence, we have $$\xi=\frac{1}{r}L-\frac{b}{rt}tF=-\frac{b}{t}K_X+\frac{1}{r}\left(\frac{t-nb}{t}\right)L=-\frac{b}{t}K_X+\left(\frac{2g-2+e}{t}\right)L\ .$$
Hence, by the Riemann--Roch theorem for vector bundles over a curve, we conclude that
$$e+n(1-g)=\chi(\mathcal{E})=\chi(\xi)=p_{(X,L)}\left(-\frac{b}{t},\frac{2g-2+e}{t}\right)=R\left(-\frac{b}{t},\frac{2g-2+e}{t}\right)\cdot n!\ .$$
This gives $\alpha r+\beta n=\frac{t}{(n-1)!}$, $\gamma=\frac{1-g}{(n-1)!}$ and $\alpha\left(-\frac{b}{t}\right)+\beta\left(\frac{2g-2+e}{t}\right)=\frac{e}{n!}$.
Let $d$ be the degree of $(X,L)$. Since $d=L^n=r^{n-1}(nb+er)$, solving the system of the first and the third equations above, we get
$$\alpha=\frac{2(g-1)}{(n-1)!}\ \ \mathrm{and} \ \ \beta=\frac{1}{n!}\left(\frac{d}{r^{n-1}}\right)\ .$$
Therefore,
\begin{equation}\label{p for P-bundles}
p_{(X,L)}(x,y)=\frac{1}{n!}\left[2n(g-1)x+\frac{d}{r^{n-1}}y+n(1-g) \right] \cdot\prod_{j=1}^{n-1}\big(ry-nx+j\big)\ .
\end{equation}
Finally, note that $rK_X+nL$ is nef if and only if $r(2g-2+e)+nb\geq 0$, equality occurring when $\text{\rm{rk}}\langle K_X,L \rangle =1$, in which case $X$ is a Fano manifold, hence $g=q=0$.}
\end{ex}

A different approach consists in reducing the computation of $p_{(X,L)}(x,y)$ to the case $r=1$. It relies on the following technical result.

\begin{lem} \label{P-fibrations}
Let $\psi: X \to Y$ be a morphism between irreducible projective
varieties with $\dim X > \dim Y$. Assume that there are positive,
coprime integers $\sigma, \tau$, and an ample line bundle $L$ on $X$ such
that $\sigma K_X+\tau L=\psi^*D$ for a line bundle $D$ on $Y$.
Let $p,q$ be two positive integers such that $\sigma p-\tau q=1$ and let $A$ be a line bundle on $Y$ such that $\sigma A+qD$ and $\tau A+pD$ are both nef. Then the following properties hold:
\begin{enumerate}
\item[$(a)$] $\mathcal{L}:=qK_X+pL+\psi^*A$ is an ample line bundle on
$X$; \item[$(b)$] $K_X+\tau\mathcal{L}
=\psi^*(\tau A+pD)$ and $\tau$ is the nefvalue of $(X,\mathcal{L})$;
\item[$(c)$] any general fiber $F$ of $\psi$ is a Fano variety and $L_{F}=\sigma\mathcal{L}_{F}$.
\end{enumerate}
In particular, assume that $\tau A+pD$ is ample. If $\psi: X \to Y$ is a surjective morphism with connected fibers, $X$ is a manifold and $Y$ is a normal variety, then $(X,\mathcal{L})$ is a Fano fibration of coindex $\dim X - \dim Y +1 - \tau .$
Moreover, if $D$ is nef then we can take $A=\mathcal{O}_X$ and in this case
$p_{(X,L)}$ and
$p_{(X,\mathcal L)}$ are related as follows:
\begin{equation}
p_{(X,L)}(x,y)=p_{(X,\mathcal L)}\left(x+\left(\frac{1-p\sigma}{\tau}\right)\frac{y}{p}\ ,\ \frac{y}{p}\right)\ , \tag{j}
\end{equation}
\begin{equation}
p_{(X,\mathcal L)}(x,y)=p_{(X,L)}\left(x-\left(\frac{1-p\sigma}{\tau}\right)y\ ,\ py\right)\ . \tag{jj}
\end{equation}
\end{lem}

\begin{proof}
Going over
the proof of \cite[Lemma 1.5.6]{BS}, note that
$$\sigma\mathcal{L}=\sigma qK_X+(1+\tau q)L+\psi^*(\sigma A)=L+\psi^*(\sigma A+qD)\ ,$$
$$K_X+\tau\mathcal{L}=K_X+\tau qK_X+\tau pL+\psi^*(\tau A)
= \sigma pK_X+\tau pL+\psi^*(\tau A)
=\psi^*(\tau A+pD)\ .$$
This gives $(a)$ and $(b)$, keeping in mind that $\tau A+pD$ is nef on $Y$. To obtain $(c)$, let $F$ be a general fiber of $\psi$. Then $-K_F=\tau\mathcal{L}_F$ and $\sigma K_F+\tau L_F=\mathcal{O}_F$.
Thus $F$ is a Fano variety and
$$\tau L_F=\sigma(-K_F)=\sigma (\tau\mathcal{L}_F)=\tau\sigma\mathcal{L}_F\ ,$$
hence $L_{F}=\sigma\mathcal{L}_{F}$. The final part of the statement follows easily from $(a)$ and $(b)$.
\end{proof}

\begin{rem}
{\em An alternative way to obtain equation \eqref{p for P-bundles} is to use Lemma \ref{P-fibrations} once we know the canonical equation of the HC of scrolls over $C$ (cf. \cite[Corollary 4.1]{L1}).
Indeed, if $X$ is a $\mathbb{P}$-bundle over $C$ with $L_F=\mathcal{O}_{\mathbb{P}^{n-1}}(r)$ for any fiber $F\cong\mathbb{P}^{n-1}$ of the projection $\pi :X\to C$,
then Lemma \ref{P-fibrations} with $(\sigma,\tau)=(r,n)$ gives an ample line bundle $\mathcal L$ on $X$ such that $(X,\mathcal L)$ is a scroll over $C$ via $\pi$.
Set $\mathcal E:=\pi_*\mathcal L$. Then $X=\mathbb{P}(\mathcal E)$ and $L\equiv r\xi+bF$ are as in Example \ref{ex scrolls over curve} with $\xi :=\mathcal L$.
Recall that the degree of the scroll $(X,\xi)$ is $\xi^n=\deg\mathcal E=:e$ while that of $(X,L)$ is $d:=L^n=r^{n-1}(re+nb)$. We have
\begin{equation}\label{ppp}
p_{(X,\xi)}(x',y')=\frac{1}{n!}\left[2n(g-1)x'+ey'+n(1-g)\right]\cdot \prod_{j=1}^{n-1}\big(y'-nx'+j\big)\ .
\end{equation}
Writing $rK_X+nL\equiv tF$ with $t$ as in Example \ref{ex scrolls over curve}, since $(t-nb)F\equiv\left(rK_X+nr\xi\right)$ we get
$$p_{(X,L)}(x,y)=\chi\left( xK_X+yL\right)=\chi\left(\overline{x}K_X+\overline{y}\xi\right)=p_{(X,\xi)}(\overline{x},\overline{y})\ ,$$
where
\begin{equation}\label{variables}
(\overline{x},\overline{y})=\left(x+\frac{bry}{t-nb}\ ,\ \frac{rt}{t-nb}y\right)\ .
\end{equation}
Replacing $(x',y')$ with $(\overline{x},\overline{y})$ expressed by \eqref{variables}, the polynomial in \eqref{ppp} gives \eqref{p for P-bundles}.
}
\end{rem}

\begin{rem}\label{rem0}
{\em From Theorem \ref{projective space,quadric,del Pezzo}\ (ii) and Example \ref{ex scrolls over curve}, we see that the two pairs $(X_1,L_1)=(\mathbb{Q}^n,\mathcal{O}_{\mathbb{Q}^n}(nb))$ and $(X_2,L_2)=(\mathbb{P}(\mathcal{O}_{\mathbb{P}^1}^{\oplus n-1}\oplus\mathcal{O}_{\mathbb{P}^1}(1)),nb\xi+bF)$, where $\xi$ is the tautological line bundle of $\mathcal{O}_{\mathbb{P}^1}^{\oplus n-1}\oplus\mathcal{O}_{\mathbb{P}^1}(1)$ and $F\cong\mathbb{P}^{n-1}$ is a fiber of $X_2\to\mathbb{P}^1$, have the same HC for any $b\in\mathbb{Z}_{\geq 1}$. This shows that in Conjecture C($n,r$) in \cite[Sec.\ 3]{L1}, the hypothesis $\text{\rm{rk}}\langle K_X,L \rangle =2$ is necessary.
More generally, referring to Conjecture C($n,r$) again, let $(X,L)$ be a polarized manifold of dimension $n\geq 2$ for which $p_{(X,L)}(x,y)$ is as in \eqref{p for P-bundles}.
If $\text{\rm{rk}}\langle K_X,L \rangle =1$ then $(X,L)=(\mathbb{Q}^n,\mathcal{O}_{\mathbb{Q}^n}(r))$. Actually,
$K_X+aL=\mathcal{O}_X$ for some $a\in\mathbb{Q}$. The expression of $p_{(X,L)}(x,y)$ shows that $p_0(r,n,0)=0$, hence $(rK_X+nL)^n=0$ by \eqref{atinfty}.
Thus $(n-ar)^nL^n=0$, that is, $a=\frac{n}{r}$ and then $rK_X+nL=\mathcal{O}_X$. Thus $X$ is a Fano manifold, whence $q=h^1(\mathcal{O}_X)=0$.
Let $H$ be the fundamental divisor on $X$: so $-K_X=\iota_XH$. By Lemma \ref{Fano}, we have $\iota_X=kn$ for some positive integer $k$.
Moreover, $k=1$ because
$1\leq k=\frac{\iota_X}{n}\leq\frac{n+1}{n}$. By \cite[Theorem 3.1.6]{BS} we conclude that $(X,H)=(\mathbb{Q}^n,\mathcal{O}_{\mathbb{Q}^n}(1))$, i.\ e.,
$(X,L)=(\mathbb{Q}^n,\mathcal{O}_{\mathbb{Q}^n}(r))$.}
\end{rem}

\medskip

The following result extends facts which are well-known for $r=1$ (see \cite[(2.12)]{F} and \cite[Proposition 3.2.1]{BS}) and $(n,r)=(3,2)$ (see \cite[Theorem $3'$]{F}).

\begin{prop}\label{P-bundles}
Let $(X,L)$ be a polarized manifold of dimension $n\geq 3$ and let $r$ be a positive integer such that $\mathrm{gcd}(r,n)=1$.
Then $X$ admits a surjective morphism $\pi: X \to C$ over a smooth curve $C$ with connected fibers such that
$(F,L_F) = (\mathbb{P}^{n-1},\mathcal{O}_{\mathbb{P}^{n-1}}(r))$ for any general fiber $F$ of $\pi$ if and only if
$X=\mathbb{P}(\mathcal{E})$ for an ample vector bundle $\mathcal{E}$ of rank $n$ on $C$ and $L \equiv r\xi+bF$, where $\xi$ is the tautological line bundle of $\mathcal{E}$ and $b$ is a suitable integer.
\end{prop}

\begin{proof}
The ``if'' part is obvious. To prove the converse, note that if
$(F,L_F)=(\mathbb{P}^{n-1},\mathcal{O}_{\mathbb{P}^{n-1}}(r))$ for any general fiber $F$ of $\pi :X\to C$, then
$(rK_X+nL)_F=rK_F+nL_F=\mathcal{O}_F$ and this gives $rK_X+nL=\pi^*D$ for some line bundle $D$ on $C$. Since $\mathrm{gcd}(r,n)=1$, we know from Lemma \ref{P-fibrations} that there exists an ample line bundle $\mathcal{L}$ on $X$ such that $(X,\mathcal{L})$ is a Fano fibration of coindex $\dim X - \dim Y +1 - n=0$.
Thus $(X,\mathcal{L})$ is a scroll over $C$ via $\pi$ with $(F,{\mathcal L}_F)= (\mathbb{P}^{n-1},\mathcal{O}_{\mathbb{P}^{n-1}}(1))$ for any fiber $F$ of $\pi$. Therefore, from \cite[(2.12)]{F} (or \cite[Proposition 3.2.1]{BS}) we deduce that $X=\mathbb{P}(\mathcal{E})$, where $\mathcal{E}=\pi_*\mathcal{L}$ is an ample vector bundle of rank $n$ on $C$; moreover, $\mathcal{L}=\xi$ and $L \equiv r\xi+bF$, as in the statement.
\end{proof}

The following result improves \cite[Proposition 4.1]{L1}
and generalizes \cite[Corollary 4.1]{L1}.

\begin{thm}\label{thm2}
Let $(X,L)$ be a polarized manifold of dimension $n\geq 2$, let $d=L^n$, $q=h^1(\mathcal{O}_X)$ and consider
a positive integer $r$ such that $\mathrm{gcd}(r,n)=1$. Suppose that either $(i)$ $q>0$, or $(ii)$ $q=0$ and $\text{\rm{rk}}\langle K_X,L \rangle =2$. Then
$X=\mathbb{P}(\mathcal{E})$ for a vector bundle $\mathcal{E}$ of rank $n$ over a smooth curve $C$ of genus $q$, and up to numerical equivalence, $L=r\xi+bF$ with $r(2q-2+e)+bn>0$, where $e:=\deg\mathcal{E}$, $\xi$ is the tautological line bundle of $\mathcal{E}$ and $F\cong\mathbb{P}^{n-1}$ is a fiber of $\mathbb{P}(\mathcal{E})\to C$ if
and only if $rK_X+nL$ is nef and
$$p_{(X,L)}(x,y):= \frac{1}{n!} \left(n(2q-2)\left(x-\frac{1}{2}\right)+ \frac{d}{r^{n-1}}y \right)\cdot \prod_{i=1}^{n-1}(ry-nx+i)\ .$$
\end{thm}

\begin{proof}
The ``only if'' part follows from Example \ref{ex scrolls over curve} (or \cite[Proposition 2.1]{L1}).
So, suppose that $p_{(X,L)}(x,y)$ is as in the statement and that $rK_X+nL$ is nef.
Since $p_0(r,n,0)=0$, we see from \eqref{atinfty} that $(rK_X+nL)^n=0$.
Thus, $rK_X+nL$ is nef but not big. Let $\varphi :X \to Y$ be the morphism as in Remark \ref{morphism} with $\dim Y<\dim X$ and
$rK_X+nL= \varphi^*D$
for some nef line bundle $D$ on $Y$.
Since $\mathrm{gcd}(r,n)=1$, by applying Lemma \ref{P-fibrations} with $\psi =\varphi$ and $(\sigma,\tau)=(r,n)$, we know that there exists an ample line bundle $\mathcal{L}$ on $X$ such that
$K_X+n\mathcal{L}$ is nef but not big. Thus, by \cite[Proposition 7.2.2]{BS}
$(X,\mathcal{L})$ is either $(\mathbb{Q}^n,\mathcal{O}_{\mathbb{Q}^n}(1))$ or a scroll over a smooth curve $C$ of genus $q$. The former case cannot occur due to our assumptions,
while in the latter
$X=\mathbb{P}(\mathcal{E})$ for a vector bundle $\mathcal{E}$ of rank $n$ on $C$ and $\mathcal{L}\equiv \xi+b'F$ for some integer $b'$, where $\xi$ is the tautological line
bundle of $\mathcal{E}$ and $F\cong\mathbb{P}^{n-1}$ is any fiber of $\mathbb{P}(\mathcal{E})\to C$. Write $L\equiv a\xi+bF$ for some integers $a,b$, and recall from Lemma \ref{P-fibrations}
 that $L_F=r\mathcal{L}_F$. Thus $a=r$, that is, $L\equiv r\xi+bF$.
Hence $rK_X+nL\equiv \left(r(2q-2+e)+bn\right)F$, where $e:=\deg\mathcal{E}$ and $r(2q-2+e)+bn\geq 0$.
Finally, note that equality cannot occur, otherwise $-rK_X\equiv nL$, hence $q=0$ and $\mathrm{rk}\langle K_X,L\rangle=1$,
contradicting $(ii)$.
\end{proof}

In particular, we have

\begin{cor}\label{C(n,r)}
Conjecture C($n,r$) in \cite[Sec.\ 3]{L1} is true
under the assumption that $rK_X+nL$ is nef.
\end{cor}

\begin{rem}\label{due rem}
{\em (i) When $r=1$ (see \cite[Corollary 4.1]{L1}), we do not need to assume that
$K_X+nL$ is nef, since this follows from adjunction theory and the fact that $\mathrm{rk}\langle K_X,L\rangle=2$.
Indeed, if $\mathrm{rk}\langle K_X,L\rangle=2$, then
$(X,L)\neq \big(\mathbb{P}^n,\mathcal{O}_{\mathbb{P}^n}(1)\big)$ and therefore $K_X+nL$ is nef
\cite[Theorem 7.2.1]{BS}. }

{\em (ii) Assume that $L$ is $r$-very ample on $X$ (see \cite[p. 225]{BS}). Then $L \gamma\geq r$ for any curve $\gamma\subset X$ (see \cite[Corollary (1.3)]{BS2}).
So, if $rK_X+nL$ is not nef, then $rK_X+(n+\epsilon)L$ is nef but not ample for some $\epsilon >0$. Thus by Mori theory, there exists
an extremal rational curve $C$ on $X$ such that $(rK_X+(n+\epsilon)L) C=0$ and $-K_X C$ is the length $\ell(R)$ of the extremal ray $R:=\mathbb{R}_{+}[C]$. This gives
$$n+1\geq -K_X C=(n+\epsilon)\ \frac{L C}{r}\geq n+\epsilon >n\ ,$$
that is, $\ell(R) = n+1$. Thus by \cite[(2.4.1)]{W} we have Pic$(X)=\mathbb{Z}$, which implies $\text{\rm{rk}}\langle K_X,L \rangle =1$. This shows that
if $L$ is $r$-very ample and $\text{\rm{rk}}\langle K_X,L \rangle =2$, then $rK_X+nL$ is nef; in particular, Conjecture C$(n,r)$ is true if $L$ is $r$-very ample.}
\end{rem}

\medskip

Another consequence of Lemma \ref{P-fibrations} is the following result for $\mathbb Q$-fibrations over smooth curves, which is known in the case of quadric fibrations, i.\ e., $r=1$ (see \cite{Fu}, or \cite[(11.8)]{F}).

\begin{prop}\label{Q-bundles}
Let $(X,L)$ be a polarized manifold of dimension $n\geq 3$ and let $r$ be a positive integer such that $\mathrm{gcd}(r,n-1)=1$.
Then $(X,L)$ is a $\mathbb Q$-fibration over a smooth curve $C$ with
$L_F= \mathcal{O}_{\mathbb{Q}^{n-1}}(r)$ for any general fiber $F \cong \mathbb Q^{n-1}$ if and only if
there exist a vector bundle $\mathcal E$ of rank $n+1$ and line bundles $\mathcal A$, $\mathcal B$ on $C$
such that $P:=\mathbb P(\mathcal E)$ contains $X$ as a smooth divisor in the linear system $|2\xi + \widetilde{\pi}^*\mathcal A|$,
where $\xi$ is the tautological line bundle on $P$, $\widetilde{\pi}:P \to C$ is the bundle projection,
and $L = (r\xi + \widetilde{\pi}^*\mathcal B)_X$.
\end{prop}

\begin{proof}
The ``if part'' is obvious, the fibration morphism being $\widetilde{\pi}|_X$. To see the converse,
let $\pi:X \to C$ be the fibration morphism. Since $\mathrm{gcd}(r,n-1)=1$, arguing as in the proof of Proposition \ref{P-bundles},
we get an ample line bundle $\mathcal L$ on $X$ such that $(X,\mathcal L)$ is a quadric fibration via $\pi$.
Then, as $n \geq 3$, the assertion follows from \cite[(11.8), case b1-Q)]{Fu} by taking $\mathcal E:= \pi_*\mathcal L$ and
noting that $\xi_X = \mathcal L$.
\end{proof}

\begin{notat}\label{data}
{\em According to Proposition \ref{Q-bundles}, a $\mathbb{Q}$-fibration $(X,L)$ over a smooth curve is described by the following data: $C, \pi, \mathcal E, \mathcal A, \mathcal B$ and $r$.
We set
$$g:=g(C),\quad e:=\deg\mathcal E,\quad a:=\deg\mathcal A,\quad b:=\deg\mathcal B\ .$$
By the canonical bundle formula for $\mathbb{P}$-bundles and adjunction, we get $K_X+(n-1)\xi_X=\pi^*(K_C+\det\mathcal E+\mathcal A)$, hence
$$rK_X+(n-1)L=\pi^*\left( r\left( K_C+\det\mathcal E+\mathcal A\right)+(n-1)\mathcal B\right)\ .$$
Therefore, $rK_X+(n-1)L\equiv tF$, where
\begin{equation}\label{t}
t:=r(2g-2+e+a)+(n-1)b\ .
\end{equation}
Clearly, if $rK_X+(n-1)L$ is nef and $\text{\rm{rk}}\langle K_X,L \rangle =2$, then $t>0$. }
\end{notat}

The following result provides a characterization of $\mathbb{Q}$-fibrations in terms of their HC, generalizing Proposition 3 and Theorem 6 of \cite{L2}.

\begin{thm}\label{thm3}
Let $(X,L)$ be a polarized manifold of dimension $n\geq 3$ with $\text{\rm{rk}}\langle K_X,L \rangle =2$ and consider a positive integer $r$ such that $\mathrm{gcd}(r,n-1)=1$. Then
$(X,L)$ is a $\mathbb{Q}$-fibration as in Notation $\ref{data}$ if and only if $rK_X+(n-1)L$ is nef and
$$p_{(X,L)}(x,y)= \frac{1}{n!}\Big[(1-n)\big(2nc+2e+(n+1)a\big)x^2+2\big((nc-(n-2)e+a)r-n(n-1)b\big)xy+$$
$$+\big((2e+a)r^2+2nbr\big)y^2 + (n-1)\big(2nc+2e+(n+1)a\big)x+$$
$$-\big((nc-(n-2)e+a)r-n(n-1)b\big)y-\frac{n(n-1)}{2}c\Big]\cdot \prod_{j=1}^{n-2} \bigg(ry-(n-1)x + j \bigg)\ ,$$
where $c:=2g-2$.
\end{thm}

\smallskip

\begin{proof}
Let $(X,L)$ be a $\mathbb{Q}$-fibration as in \ref{data}. Then $(X,\xi_X)$ is a quadric fibration over $C$.
The same equation as in \cite[Proposition 3]{L2}, rewritten in terms of coordinates $(x',y')=(\frac{1}{2}+u,v)$ is the following (taking into account that $b$ in \cite{L2} is our $-a$):
\begin{eqnarray} \label{expr1}
p_{(X,\xi_X)}(x',y') &=& \frac{1}{n!} \Bigg[\ (1-n)\bigg(2nc+2e+(n+1)a\bigg){x'}^2 \\ \nonumber
& &  +\ 2\bigg(nc-(n-2)e+a\bigg) x'y' +(2e+a){y'}^2 \\ \nonumber
& &  +\ (n-1)\bigg(2nc+2e+(n+1)a\bigg)x' - \bigg(nc-(n-2)e+a\bigg)y' \\ \nonumber
& &  -\ \binom{n}{2}c \ \Bigg]\cdot \prod_{i=1}^{n-2} \bigg(y'-(n-1)x' + i \bigg)=0.
\end{eqnarray}
Recalling that $L\equiv r\xi_X+bF$ and \ref{data}, we have $(t-(n-1)b)F\equiv r(K_X+(n-1)\xi_X)$. Thus
$$F\equiv \frac{r}{t-(n-1)b}(K_X+(n-1)\xi_X)\ ,$$
and substituting this expression of $F$ into $xK_X+y(r\xi_X+bF)$, we get
$$p_{(X,L)}(x,y)=\chi(xK_X+yL)=\chi(\overline{x}K_X+\overline{y}\xi_X)=p_{(X,\xi_X)}(\overline{x},\overline{y})\ ,$$
where
$$(\overline{x},\overline{y})=\left(x+\frac{br}{t-(n-1)b}\ y\ ,\ \frac{tr}{t-(n-1)b}\ y \ \right)$$
Then the expression of $p_{(X,L)}(x,y)$ follows from \eqref{expr1}
by replacing $(x',y')$ with $(\overline{x},\overline{y})$ as
above and taking into account \eqref{t}.

To prove the converse, let $p_{(X,L)}(x,y)$ be as in the statement and write it as
$R(x,y)\cdot\prod_{j=1}^{n-2} \big(ry-(n-1)x+j \big)$. Then
$R(x,y)=Ax^2+Bxy+Cy^2+Ex+Gy+H$, where $A:=\frac{1}{n!}\left[(1-n)(2nc+2e+(n+1)a) \right] ,
B:=\frac{1}{n!}\left[ -2enr+2cnr+2bn-2n^2b+4re+2ar\right]$ and $C:=\frac{1}{n!}\left[ 2er^2+r^2a+2nbr\right]$.
Recalling \eqref{atinfty}, from the equality
$$\frac{1}{n!}(K_X+yL)^n=p_0(1,y,0)=R_0(1,y,0)\cdot \prod_{j=1}^{n-2} \bigg(ry-(n-1)\bigg)=
(A+By+Cy^2)\cdot \left[ry+(1-n) \right]^{n-2}\ ,$$
where $R_0(x,y,z)$ is the homogeneous polynomial associated with $R(x,y)$,
we deduce that
$$L^n=Cn!r^{n-2},\quad K_XL^{n-1}=(n-1)!\big(Br+C(n-2)(1-n)\big)r^{n-3},$$
$$K_X^2L^{n-2}=2(n-2)!\left(Ar^2+B(n-2)r(1-n)+
C\frac{(n-2)(n-3)}{2}(1-n)^2\right)r^{n-4}\ .$$
A computation with Maple shows that
$(rK_X+(n-1)L)^2L^{n-2}=r^2K_X^2L^{n-2}+2r(n-1)K_XL^{n-1}+(n-1)^2L^n=0$. On the other hand,
$\frac{1}{n!}(rK_X+(n-1)L)^n=p_0(r,n-1,0)=0$ by \eqref{atinfty}. Therefore
$$(rK_X+(n-1)L)^n=0\quad\quad \mathrm{and}\quad \quad (rK_X+(n-1)L)^2L^{n-2}=0\ .$$
Since $rK_X+(n-1)L$ is nef, by applying Remark \ref{morphism} we see that the morphism $\varphi$ has a one dimensional image, i.\ e., $Y$ is a smooth curve.
Thus by Lemma \ref{P-fibrations} with $(\sigma,\tau)=(r,n-1)$, there exists an ample line bundle $\mathcal L$ on $X$ such that $K_X+(n-1)\mathcal L=p(rK_X+(n-1)L)=\varphi^*D$ for some ample
line bundle $D$ on $Y$. Since $\text{\rm{rk}}\langle K_X,\mathcal L \rangle =\text{\rm{rk}}\langle K_X,L \rangle =2$, we conclude that
$(X,\mathcal L)$ is a quadric fibration over $Y$ \cite[$\S\S 7.2, 7.3$]{BS}. Then Lemma \ref{P-fibrations} $(c)$ allows us to conclude.
\end{proof}

\begin{rem}\label{tre rem}
{\em (j) An alternative way to get the expression of $p_{(X,L)}(x,y)$ in Theorem \ref{thm3} is to follow the method outlined at the beginning of this section
summarized by Algorithm $2$ in the Appendix.}

{\em (jj) When $r=1$ (see \cite[Theorem 6]{L2}), we do not need to assume that
$K_X+(n-1)L$ is nef, provided that $(X,L)$ is not a scroll over a smooth curve and $\mathrm{rk}\langle K_X,L\rangle=2$.
Actually, under these assumptions, this property comes from \cite[Proposition 7.2.2 and Theorem 7.2.4]{BS}.}

{\em (jjj) Assume that $L$ is $r$-very ample on $X$ and argue as in Remark \ref{due rem}.
If $rK_X+(n-1)L$ is not nef, then $rK_X+(n-1+\epsilon)L$ is nef but not ample for some $\epsilon >0$. Hence by Mori theory and \cite[Lemma 6.4.2]{BS} there exists
an extremal rational curve $C$ on $X$ such that $(rK_X+(n-1+\epsilon)L) C=0$ and $-K_X C$ is the length $\ell(R)$ of the extremal ray $R:=\mathbb{R}_{+}[C]$. This gives
$$n+1\geq -K_X C=(n-1+\epsilon)\ \frac{L C}{r}\geq n-1+\epsilon > n-1\ ,$$
that is, $\ell(R) = n$ or $n+1$. Therefore by \cite[(2.4)]{W} we have either
Pic$(X)=\mathbb{Z}$, which implies $\text{\rm{rk}}\langle K_X,L \rangle =1$,
or Pic$(X)=\mathbb{Z}\oplus\mathbb{Z}$ and the contraction of $R$ defines a morphism
$\rho_R:X\to B$ onto a smooth curve $B$
whose general fiber $F$ is a smooth Fano manifold with Pic$(F)=\mathbb{Z}$.
Let $\Gamma$ be any rational curve on $F$. Then the nefvalue morphism associated to $(X,L)$
with nefvalue $\frac{n-1+\epsilon}{r}$ contracts $\Gamma$, hence $\left( rK_X+(n-1+\epsilon)L\right)\Gamma=0$.
This shows that $-K_F \Gamma=-K_X \Gamma=(n-1+\epsilon)\ \frac{L \Gamma}{r}> n-1$, i.\ e., $-K_F \Gamma\geq n=\dim F+1$.
Thus by \cite[Theorem 6.3.14]{BS} we get $F\cong\mathbb{P}^{n-1}$. Since distinct general fibers of $\rho_R$ are numerically equivalent,
from Proposition \ref{P-bundles} we deduce that $X=\mathbb{P}(\mathcal{V})$ for a vector bundle $\mathcal{V}$ of rank $n$ on $B$.
The above argument shows that if $L$ is $r$-very ample on $X$ with $\mathrm{gcd}(r,n-1)=1$, $\text{\rm{rk}}\langle K_X,L \rangle =2$
and $X$ is not a $\mathbb{P}^{n-1}$-bundle over a smooth curve, then $rK_X+(n-1)L$ is nef.}

{\em (jv) Taking $r=1$ and $b=0$ in Theorem \ref{thm3}, the polynomial $p_{(X,L)}(x,y)$ coincides with that given in \cite[Proposition 3]{L2} in
terms of coordinates $(x,y)$ instead of $(u,v)$.}
\end{rem}

Finally, let us give here also a characterization of Fano fibrations of coindex $2$ over smooth curves in the case $r=1$.

\begin{thm}\label{expr_with_r=1}
Let $(X,L)$ be a polarized manifold of dimension $n\geq 3$.
Suppose that $(X,L)$ is a Fano fibration of coindex $2$ over a smooth curve of genus $g$ such that $K_X+(n-2)L\equiv tF$ for some positive integer $t$,
where $F$ is a general fiber of $X$. Then
$$p_{(X,L)}(x,y)=\left(\sum_{0\leq i+j\leq 3}c_{ij}x^iy^j\right)\cdot \prod_{i=1}^{n-3} \bigg(y-(n-2)x + i \bigg)\ ,$$
with
\begin{eqnarray*}
c_{30} & = & (n-2)^2\Bigg(\frac{t\delta}{(n-1)!}-(n-2)\frac{d}{n!}\Bigg) \ , \\
c_{21} & = & -(n-2)\Bigg(\frac{2t\delta}{(n-1)!}-3(n-2)\frac{d}{n!}\Bigg)  \ , \\
c_{12} & = & \frac{t\delta}{(n-1)!}-3(n-2)\frac{d}{n!}   \ ,  \\
c_{03} & = & \frac{d}{n!}\ , \\
c_{20} & = & -\frac{3}{2}(n-2)^2\Bigg( \frac{t\delta}{(n-1)!}-(n-2)\frac{d}{n!}\Bigg)   , \\
c_{11} & = & (n-2)\Bigg( \frac{2t\delta}{(n-1)!} - 3(n-2)\frac{d}{n!}\Bigg)  \ , \\
c_{02} & = & -\frac{1}{2}\Bigg( \frac{t\delta}{(n-1)!} - 3(n-2)\frac{d}{n!}\Bigg) \ ,
\end{eqnarray*}
\begin{eqnarray*}
c_{10} & = & \frac{1}{2}(n-2)^2\Bigg( \frac{t\delta}{(n-1)!} - (n-2)\frac{d}{n!}\Bigg)+2\frac{g-1}{(n-3)!} \ , \\
c_{01} & = & \frac{\chi}{(n-2)!}+\frac{1}{2}\frac{t\delta}{(n-1)!}-\frac{d}{2(n!)}(3n-4)+\frac{g-1}{(n-3)!} \ , \\
c_{00} & = & -\frac{g-1}{(n-3)!} \ ,
\end{eqnarray*}
where $d:=L^n$, $\delta:=FL^{n-1}$ and
$\chi:=\chi(L)=\frac{1}{2}\left(d-t\delta+2t\right)-n(g-1)$.
Conversely, assume that $\text{\rm{rk}}\langle K_X,L \rangle =2$
and $K_X+(n-2)L$ is nef. If $p_{(X,L)}(x,y)$ is as above for
integers $d, \delta, \chi, t, g$ such that
$\chi=\frac{1}{2}\left(d-t\delta+2t\right)-n(g-1)$, then
$(X,L)$ is a Fano fibration of coindex $2$ over a smooth curve of
genus $g$ such that $K_X+(n-2)L\equiv tF$, $\chi(L)=\chi$, $L^n=d$
and $FL^{n-1}=\delta$ for any general fiber $F$ of $X$.
\end{thm}

\begin{proof}
Let $(X,L)$ be a Fano fibration of coindex $2$ over a smooth curve $C$ via a morphism $\varphi :X\to C$ and let $F$ be a general fiber. By
\cite[Theorem 6.1]{BLS}, we know that
$$p_{(X,L)}(x,y)=R(x,y)\cdot\prod_{i=1}^{\iota_F-1}\big(y-\iota_Fx+i\big)=R(x,y)\cdot\prod_{i=1}^{n-3}\Big(y-(n-2)x+i\Big)\ ,$$
where $R(x,y)=\sum_{0\leq i+j\leq 3}c_{ij}x^iy^j$ for some $c_{ij}\in\mathbb{Q}$. Since $R(x,y)=-R(1-x,-y)$ we have the relations
\begin{equation}\label{rel}
c_{30} = 4c_{00}+2c_{10}\ , c_{21} = - c_{11}\ , c_{12} = -2c_{02} \ , c_{20} = -6c_{00}-3c_{10}\ .
\end{equation}
By Theorem \ref{projective space,quadric,del Pezzo} applied to the pair $(F,L_F)$, we deduce that
{\small
$$p_{(F,L_F)}(x,y)=\Bigg(\frac{\delta}{(n-1)!}\left(y-(n-2)x\right)^2 + \frac{(n-2)\delta}{(n-1)!}\left(y-(n-2)x\right) +
\frac{1}{(n-3)!}\Bigg)\cdot \prod_{i=1}^{n-3}\big(y-(n-2)x+i\big)\ .$$}

\noindent Now apply $(\ref{R_F1})$ and Algorithm $2$, noting that the polynomial $R_F(x,y)$
appearing in $(\ref{R_F1})$ is just the first factor of $p_{(F,L_F)}(x,y)$. Next the use of the following relations
$$1-g=\chi(\mathcal{O}_C)=\chi(\mathcal{O}_X)=p_{(X,L)}(0,0),\ \chi(L)=p_{(X,L)}(0,1),\ \frac{L^n}{n!}=p_0(0,1,0), $$
coming from $\varphi_*\mathcal{O}_X=\mathcal{O}_C$, the projection
formula and \eqref{atinfty}, allow us to express
the $c_{ij}$'s in terms of $d, \delta, t, \chi, g$. Finally, from
the relation
$$t+1-g=\chi(\varphi_*(tF))=\chi(tF)=p_{(X,L)}(1,n-2)=R(1,n-2)(n-3)!\
,$$ by using Maple we get
$\chi:=\chi(L)=\frac{1}{2}\left(d-t\delta+2t\right)-n(g-1)$.
This gives the first part of the statement.

\medskip

Now suppose that $(X,L)$ is a polarized manifold of dimension
$n\geq 3$ with $\text{\rm{rk}}\langle K_X,L \rangle =2$, such that
$p_{(X,L)}(x,y)$ is as in the statement for some integers $d,
\delta, \chi, t, g$ such that
$\chi=\frac{1}{2}\left(d-t\delta+2t\right)-n(g-1)$. The
expression of $p_{(X,L)}(x,y)$ combined with \eqref{atinfty} shows
that $(K_X+(n-2)L)^n=p_0(1,n-2,0)=0$. Thus, $K_X+(n-2)L$ is nef
but not big. By Remark \ref{morphism}, there exists a morphism
$\varphi :X \to Y$ onto a normal variety $Y$ with $\dim Y<\dim X$
such that $K_X+(n-2)L= \varphi^*D$ for some nef line bundle $D$ on
$Y$. Recalling \eqref{atinfty}, we have
$$p_0(x,1,0)=\frac{1}{n!}\big( xK_X+L\big )^{n}=\frac{1}{n!}\left(L^n+\binom{n}{1}K_XL^{n-1}x+\binom{n}{2}K_X^2L^{n-2}x^2+\dots\ \right).$$
On the other hand,
$$p_0(x,1,0)=(c_{30}x^3+c_{21}x^2+c_{12}x+c_{03})\cdot \prod_{i=1}^{n-3}\big(1-(n-2)x\big)=$$
$$=(-1)^{n-3}(c_{30}x^3+c_{21}x^2+c_{12}x+c_{03})\cdot \big(
(n-2)x-1\big)^{n-3}=$$
$$=c_{03}+(-1)^{n-3}\left[ c_{03}\binom{n-3}{1}(n-2)(-1)^n+c_{12}(-1)^{n-1}\right] x\ +$$
$$+\ \left[c_{21}-c_{12}\binom{n-3}{1}(n-2)\ +\ c_{03}\binom{n-3}{2}(n-2)^{2}\right] x^2+\dots\ .$$
Comparing the coefficients, we get
$$L^n\ =\ n!\ c_{03} \ , \quad K_XL^{n-1}\ =\ (-1)^{n-3}(n-1)!\left[c_{03}\binom{n-3}{1}(n-2)(-1)^n+c_{12}(-1)^{n-1}\right],$$
$$K_X^2L^{n-2}\ =\  2(n-2)!\left[c_{21}-c_{12}\binom{n-3}{1}(n-2)+c_{03}\binom{n-3}{2}(n-2)^2\right] $$
and then a computation with Maple shows that
$$(K_X+(n-2)L)^2L^{n-2}=K_X^2L^{n-2}+2(n-2)K_XL^{n-1}+(n-2)^2L^n=0 \ .$$
Therefore, $\dim Y\leq 1$. Since $\text{\rm{rk}}\langle K_X,L \rangle =2$, this implies $\dim Y=1$ and the fact that $D$ is ample.
Thus $(X,L)$ is a Fano fibration of coindex $2$ over the smooth curve $Y$ whose genus $g(Y)$ is $g$, because
$$1-g(Y)=\chi(\mathcal{O}_Y)=\chi(\mathcal{O}_X)=p_{(X,L)}(0,0)=(n-3)!\ c_{00}= 1-g\ .$$
Moreover, writing $K_X+(n-2)L\equiv t'F$ for some positive integer
$t'$, where $F$ is a general fiber of $\varphi$, by the relation
$\chi=\frac{1}{2}\left(d-t\delta+2t\right)-n(g-1)$ we see that
$$t'+1-g=\chi (t'F) = \chi (K_X+(n-2)L) = p_{(X,L)}(1,n-2) = t+1-g \ ,$$ i.\ e., $t'=t$.
Finally, in view of the above expressions of $L^n$ and $K_XL^{n-1}$ we have $FL^{n-1}=\frac{1}{t}\left(K_XL^{n-1}+(n-2)L^n\right)=\delta$.
\end{proof}

\begin{rem}
{\em A result similar to Theorem \ref{expr_with_r=1} can be obtained for $rK_X+(n-2)L\equiv tF$ with $r$ a positive integer such that
$\mathrm{gcd}(r,n-2)=1$ by using Lemma \ref{P-fibrations}. Notice that the corresponding expression of $p_{(X,L)}(x,y)$ is very intricate.}
\end{rem}

\section{Case $\mathrm{rk}\langle K_X,L\rangle=2$: Fano fibrations over varieties}\label{subsection2}

Here we describe a procedure to obtain the canonical equation of the
Hilbert curve for Fano fibrations of low coindex over a
normal variety of dimension $m\geq 2$.

More precisely, let $X$ be a manifold of dimension $n$ and let $\pi :X\to Y$ be a morphism onto a normal variety $Y$ of dimension $m<n$. Let $L$ be an ample line bundle on $X$ such that $rK_X+\iota_FL=\pi^*A$ for some ample line bundle $A$ on $Y$. Then there exists an integer $s>>0$ such that $sA$ is very ample on $Y$. Thus $s(rK_X+\iota_FL)=\pi^*sA$ is spanned on $X$ and then we can take a smooth irreducible element $V\in |\pi^*sA|$. Consider the following exact sequence
$$0\to xK_X+yL+(x-1)V \to xK_X+yL+xV\to xK_V+yL_V\to 0\ .$$
Since $V \in | srK_X+s\iota_FL |$, we have
\begin{equation}\label{p}
p_{(X,L)}\big(x(sr+1),y+sx\iota_F\big)=p_{(X,L)}\big((x-1)sr+x,y+s\iota_F(x-1)\big)+p_{(V,L_V)}\left(x,y\right)\ ,
\end{equation}
with $rK_V+\iota_FL_V={\pi_{|V}}^*(rs+1)A$. By using general coordinates $(x',y')$, we can write
$$p_{(X,L)}\big(x',y'\big)=R_X(x',y')\cdot\prod_{i=1}^{\iota_F-1}\big( ry'-\iota_Fx'+i \big)\ ,$$
for some polynomial $R_X$ of degree $n-\iota_F+1$.
Letting $(x',y')=(x(sr+1),y+sx\iota_F)$ and $((x-1)sr+x,y+s\iota_F(x-1))$ respectively, we obtain
$$p_{(X,L)}\big(x(sr+1),y+sx\iota_F\big)=R_X(x(sr+1),y+sx\iota_F)\cdot\prod_{i=1}^{\iota_F-1}\big( ry- \iota_Fx + i \big)\ ,$$
$$p_{(X,L)}\big((x-1)sr+x,y+s\iota_F(x-1)\big)=R_X((x-1)sr+x,y+s\iota_F(x-1))\cdot\prod_{i=1}^{\iota_F-1}\big( ry-\iota_Fx+i \big)\ .$$
Similarly, for the pair $(V,L_V)$ we have
$$p_{(V,L_V)}\big(x,y\big)=R_V(x,y)\cdot\prod_{i=1}^{\iota_F-1}\big( ry-\iota_Fx+i \big)\ .$$
Thus (\ref{p}) gives
\begin{equation}\label{R}
R_X\big(x+srx,y+s\iota_Fx\big)=R_X\big((x+srx)-sr,(y+s\iota_Fx)-s\iota_F\big)+R_V\big(x,y\big)\ .
\end{equation}
Set $M:=\left( \begin{array}{ c c }
sr+1 & 0 \\
s\iota_F & 1 \\
\end{array}\right)$, $\vec{v}:=\left( -sr , -s\iota_F \right)$ and $\vec{x}:=\left( x , y \right)$. Then (\ref{R})
becomes
\begin{equation}\label{RR}
R_X\big(\vec{x}M)=R_X\big(\vec{x}M+\vec{v}\big)+R_V\big(\vec{x}\big)\ .
\end{equation}
Letting $X_0:=X$ and $X_1:=V$, equation \eqref{RR} can be rewritten as
$$R_{X_0}\big(\vec{x}M)=R_{X_0}\big(\vec{x}M+\vec{v}\big)+R_{X_1}\big(\vec{x}\big)\ .$$
For any $j\in\{0,...,m-2\}$, denote by $X_j$ the pull-back via $\pi$ of the transverse intersection of $j$ general elements
of $|sA|$. Then by an inductive argument we obtain $rK_{X_j}+\iota_FL_{X_j}={\pi_{|X_j}}^*(jrs+1)A$ and
\begin{equation}\label{RRR}
R_{X_{j}}\big(\vec{x}M)=R_{X_{j}}\big(\vec{x}M+\vec{v}\big)+R_{X_{j+1}}\big(\vec{x}\big)
\end{equation}
for $j\in\{1,...,m-2\}$.
Equation (\ref{RRR}) says that if we know the term $R_{X_{j+1}}\big(\vec{x}\big)$,
it is possible to go back to the term $R_{X_{j}}\big(\vec{x}\big)$, and so on.

\medskip

Finally, consider the case $j=m-1$ and for simplicity set $W:=X_{m-1}$ and $R\big(\vec{x}\big):=R_{W}\big(\vec{x}\big)$.
Then $p_{(W,L_W)}(\vec{x})=R(\vec{x})\cdot \prod_{k=1}^{\iota_F-1}\big( ry-\iota_Fx+k \big)$.
Note that $W$ is a smooth variety of dimension $n-m+1$ endowed with a morphism
$\varphi:=\pi_{|X_{m-1}} :W\to C$ onto a smooth irreducible curve $C$ which is the transverse intersection of $m-1$ general elements of $|sA|$.
Then $rK_W+\iota_FL_W=\varphi^*\mathcal A$ for some ample line bundle $\mathcal A$ on $C$. Hence $rK_W+\iota_FL_W\equiv tF$ for some positive integer $t$.
Thus, for a general fiber $F$ of $\varphi$, by considering the following exact sequences
$$0\to xK_W+yL_W - F \to xK_W+yL_W\to xK_F+yL_F\to 0\ ,$$
$$0\to xK_W+yL_W - 2F \to xK_W+yL_W-F\to xK_F+yL_F\to 0\ ,$$
$$\dots$$
$$0\to xK_W+yL_W - tF = (x-r)K_W+(y-\iota_F)L_W \to xK_W+yL_W - (t-1)F\to xK_F+yL_F\to 0\ ,$$
we obtain that
\begin{equation}\label{RRRR}
p_{(W,L_W)}(x,y)=p_{(W,L_W)}(x-r,y-\iota_F)+t\ p_{(F,L_F)}(x,y)\ .
\end{equation}

\medskip

\noindent Since $F$ is a Fano manifold with $rK_F+\iota_FL_F= \mathcal{O}_F$, by Proposition \ref{prop1} we know that
$$p_{(F,L_F)}(x,y)=R_F(x,y)\cdot\prod_{i=1}^{\iota_F-1}\big(ry-\iota_Fx+i\big)\ ,$$
for a suitable polynomial $R_F$ of degree $n-m-\iota_F+1$.
Then equation \eqref{RRRR} with the above relations yields
\begin{equation}\label{R_F}
R(x,y)=R(x-r,y-\iota_F)+t\ R_F(x,y).
\end{equation}
Therefore, once we know $R_F(x,y)$, from (\ref{R_F}) we can find $R(\vec{x}):=R(x,y)$.
In conclusion, by (\ref{RRR}), going back by induction, we can obtain all the $R_{X_j}\big(\vec{x}\big)$'s.
This completes the procedure (see Algorithm \ref{Alg3}, in the Appendix).

\begin{rem}\label{remark}
{\em A Fano fibration of dimension $n$ and coindex $0$ over a normal surface is in fact a projective bundle $\mathbb{P}(\mathcal{V})$ for some ample vector bundle $\mathcal{V}$ of rank $n-1$ over a smooth surface in view of \cite[(3.2.1)]{BSW} and \cite[Lemma (2.12)]{F}.}
\end{rem}

As an example, we apply the above method to find the Hilbert curve of a scroll over a smooth surface
with the help of Maple.

\begin{ex}\label{ex scroll over S}
{\em Let $X:=\mathbb{P}(\mathcal{V})$ for some ample vector bundle $\mathcal{V}$ of rank $n-1$ over a smooth surface $S$.
Let $L$ be the tautological line bundle and let $F\cong\mathbb{P}^{n-2}$ be a fiber of the bundle projection $\pi:X\to S$.
In this case $\iota_F=n-1$. So, letting $r=1$, we have $$p_{(X,L)}(x,y)=R(x,y)\cdot\prod_{i=1}^{\iota_F-1}\big(ry-\iota_Fx+i\big)=R(x,y)\cdot\prod_{i=1}^{n-2}\big[y-(n-1)x+i\big]\ ,$$ where,
due to the invariance of $p_{(X,L)}$ under the Serre involution, the polynomial $R(x,y)$ has the following expression:
\begin{equation}\label{R(x,y)}
R(x,y)=a_{11} x^2+2a_{12} xy+a_{22} y^2-a_{11} x-a_{12} y+a_{00}\ ,
\end{equation}
with $a_{00}=\frac{\chi(\mathcal{O}_X)}{(n-2)!}$, for some $a_{11}, a_{12}, a_{22}\in\mathbb{Q}$.
Since we are assuming that $(X,L)$ is a scroll over $S$, we have $K_X+(n-1)L=\pi^*A$ for some ample line bundle $A$ on $S$.
Moreover,
\begin{equation}\label{A}
A:=K_S+\det\mathcal{V}
\end{equation}
by the canonical line bundle formula. Let $s$ be a positive integer such that $sA$ is very ample. Let $C\in |sA|$ be any smooth curve and let
$V=\pi^{-1}(C)$. In the present case \eqref{R} becomes
$$R\big((1+s)x,y+s(n-1)x\big)=R\big((1+s)x-s,(y+s(n-1)x)-s(n-1)\big)+R_{(V,L_V)}\big(x,y\big)\ .$$
Note that $(V,L_V)$ is a scroll over $C$ via $\pi_{|V}:V\to C$. Then by Theorem \ref{thm2} we know that
$$R_{(V,L_V)}\big(x,y\big)=\frac{2q-2}{(n-2)!}\ x + \frac{d}{(n-1)!}\ y-\frac{q-1}{(n-2)!}\ .$$
Here $q=g(C)=1+\frac{s}{2}(K_SA+sA^2)$ by the genus formula and
\begin{equation}\label{d}
d=L_V^{n-1}=(\det\mathcal{V})C=sA(A-K_S)
\end{equation}
by the Chern-Wu relation and
\eqref{A}. So, by using Maple and comparing $R\big((1+s)x,y+s(n-1)x\big)-R\big((1+s)x-s,(y+s(n-1)x)-s(n-1)\big)$ with $R_{(V,L_V)}\big(x,y\big)$, we obtain
{\small{
\begin{align*}
a_{11} & = \frac{n-1}{2(s^2+s)n!}\big\{\left[-2\chi(\mathcal{O}_S)(n-1)^2+2\chi(L)(n-1)\right](s^2+s)-d(n+1)s+n(2q-2)-d\big\} \qquad \\
& = \frac{n-1}{2n!}\left[-2\chi(\mathcal{O}_S)(n-1)^2+2\chi(L)(n-1)+(n+1)K_SA-A^2\right]\ , \\
a_{12} & = \frac{1}{2s(n!)}\big\{\left[2\chi(\mathcal{O}_S)(n-1)^2-2\chi(L)(n-1)\right]s+d\big\} \\
& = \frac{1}{2(n!)}\left[2\chi(\mathcal{O}_S)(n-1)^2-2\chi(L)(n-1)-K_SA+A^2\right]\ , \\
a_{22} & =\frac{1}{2s(n!)}\big\{\left[-2\chi(\mathcal{O}_S)(n-1)+2\chi(L)\right]s+d\big\}=\frac{1}{2(n!)}\left[-2\chi(\mathcal{O}_S)(n-1)+2\chi(L)-K_SA+A^2\right]\ .
\end{align*}}}

\noindent Note that to get the final expressions of the $a_{ij}$'s we used \eqref{d} and the fact that $-d(n+1)s+n(2q-2)-d=(s^2+s)[(n+1)K_SA-A^2]$.}
\end{ex}

\medskip

Actually, Fano fibrations of coindex $0$ over a smooth surface can be characterized by means of their HC.

\begin{thm}\label{thm scrolls over S}
Let $(X,L)$ be a polarized manifold of dimension $n\geq 3$. If
$(X,L)$ is a Fano fibration of coindex $0$
over a smooth surface $S$
then
\begin{eqnarray*}\label{scrolls S r=1}
p_{(X,L)}(x,y) &=& \frac{1}{n!} \Bigg\{\  \frac{n-1}{2}\left[-2\chi_0(n-1)^2+2\chi(n-1)+(n+1)k-h\right] x^2 \ + \\ \nonumber
& & \quad \quad + \ \left[2\chi_0(n-1)^2-2\chi (n-1)-k+h\right] xy \ + \\ \nonumber
& & \quad \quad + \ \frac{1}{2}\left[-2\chi_0 (n-1)+2\chi -k+h\right] y^2 \ + \\ \nonumber
& & \quad \quad + \ \frac{n-1}{2}\left[2\chi_0 (n-1)^2-2\chi (n-1)-(n+1)k+h\right] x \ + \\ \nonumber
& & \quad \quad + \ \frac{1}{2}\left[-2\chi_0 (n-1)^2+2\chi (n-1)+k-h\right] y \ + \\ \nonumber
& & \quad \quad + \ n(n-1)\chi_0 \ \Bigg\} \cdot\prod_{i=1}^{n-2}\bigg(y-(n-1)x+i\bigg)\ , \nonumber
\end{eqnarray*}
where $\chi_0:=\chi(\mathcal{O}_S), \chi:=\chi(L), k:=K_SA$ and $h:=A^2$, $A$ being an ample line bundle on $S$ such that $K_X+(n-1)L=\pi^*A$.
Conversely, assume that $\text{\rm{rk}}\langle K_X,L \rangle =2$ and $K_X+(n-1)L$ is nef. If $p_{(X,L)}(x,y)$ is as above for some integers $\chi_0, \chi, k, h$ with $h>0$, then $(X,L)$ is a Fano fibration of coindex $0$ over a smooth surface.
\end{thm}

\begin{proof}
Keeping in mind Remark \ref{remark}, the ``only if'' part of the statement follows from Example \ref{ex scroll over S} once we consider that
$X=\mathbb{P}(\mathcal{E})$ for some ample vector bundle $\mathcal{E}$ of rank $n-1$ on $S$, $L$ being the tautological line bundle, and $A=K_S+\det\mathcal{E}$.
Thus assume that $(X,L)$ is a polarized manifold of dimension $n\geq 3$ with $\text{\rm{rk}}\langle K_X,L \rangle =2$ for which $K_X+(n-1)L$ is nef and $p_{(X,L)}(x,y)$ is as in the statement. Note that
$(K_X+(n-1)L)^n=n!p_{0}(1,n-1,0)=0$ by \eqref{atinfty}.
Hence $K_X+(n-1)L$ is nef but not big. Thus by Remark \ref{morphism} there exists a morphism $\varphi :X \to Y$ onto a normal variety $Y$ with $\dim Y<n$ such that
\begin{equation}\label{aF}
K_X+(n-1)L= \varphi^*D
\end{equation}
for some nef line bundle $D$ on $Y$. Write $p_{(X,L)}(x,y)=R(x,y)\cdot\prod_{i=1}^{n-2}\big(y-(n-1)x+i\big)$, where $R(x,y)$ is as in \eqref{R(x,y)} of Example \ref{ex scroll over S}. From
$$\frac{1}{n!}\left(xK_X+L\right)^n=p_{0}(x,1,0)=\left(a_{11} x^2+2a_{12} x+a_{22}\right)\cdot \left[ (n-1)x-1\right]^{n-2},$$
it follows that
$$\frac{1}{n!}\left[L^n+\binom{n}{1}K_XL^{n-1}x+\binom{n}{2}K_X^2L^{n-2}x^2+\binom{n}{3}K_X^3L^{n-3}x^3+ \dots\right]= $$
$$=a_{22} (-1)^n+\left[ a_{22}\binom{n-2}{1}(-1)^{n-1}(n-1)+2a_{12}(-1)^n\right]x \ +$$
$$+\ \left[ a_{22}\binom{n-2}{2}(-1)^n(n-1)^2+2a_{12}\binom{n-2}{1}(-1)^{n-1}(n-1)+a_{11}(-1)^n \right]x^2\ + $$
$$+\ (-1)^{n-1}\left[ a_{22}\binom{n-2}{3}(n-1)^3-2a_{12}\binom{n-2}{2}(n-1)^2+a_{11}\binom{n-2}{1}(n-1) \right] x^3 \ + \dots \ .$$
Comparing the coefficients of $x^2, xy$ and $y^2$, a computation with Maple shows that
$$(K_X+(n-1)L)^3L^{n-3}=K_X^3L^{n-3}+3(n-1)K_X^2L^{n-2}+3(n-1)^2K_XL^{n-1}+(n-1)^3L^n = 0,$$ hence $\dim Y\leq 2$.
Note that $\dim Y=1$ or $2$, because $\text{\rm{rk}}\langle K_X,L \rangle =2$.
Since the nefvalue of $(X,L)$ is $n-1$ in view of \eqref{aF}, we deduce by
\cite[Theorem 7.3.2]{BS} and \cite[(11.8)]{Fu} that $(X,L)$ is either (i) a scroll
over a smooth surface or (ii) a quadric fibration over a smooth curve of genus $q$.
Assume we are in case (ii). Then comparing the coefficients of $x^2, xy, y^2$ in $R(x,y)$ of $p_{(X,L)}$ given in the statement
with those provided by Theorem \ref{thm3} for $r=1$ (see formula \eqref{expr1}), we get
the following equalities:
\begin{align*}
n! a_{11} & =  \frac{n-1}{2}\left[ -2\chi_0(n-1)^2+2\chi (n-1)+(n+1)k-h\right] \\
& =  (1-n)\left[ 2n(2q-2)+2e+(n+1)a\right]\ , \\
n! a_{12} & =  \frac{1}{2}\left[ 2\chi_0 (n-1)^2-2\chi (n-1)-k+h\right]\\
& =  n(2q-2)-(n-2)e+a\ , \\
n! a_{22} & =  \frac{1}{2}\left[-2\chi_0 (n-1)+2\chi -k+h\right] \\
& =  2e+a\ .
\end{align*}

\noindent A check with Maple shows that the above three equations imply $h=0$, but this is impossible because $h$ is assumed to be a positive integer. Thus $(X,L)$ is as in case (i), i.\ e. $(X,L)$ is a Fano fibration of coindex $0$ over a  smooth surface.
\end{proof}

\smallskip

A result similar to Theorem \ref{thm scrolls over S} holds also for $rK_X+(n-1)L$
with $\mathrm{gcd}(r,n-1)=1$. Finally, summing-up the above results
and the proof of \cite[Proposition 5.1]{L1},
we can deduce also the following result comparable with Corollary \ref{C(n,r)} for $(n,r)=(3,2)$.

\begin{cor}
The Conjecture $C(3,2)$ stated in \cite{L1} is true provided that $(X,L)$ does not contain $(-1)$-planes.
\end{cor}

\section{Appendix}\label{Appendix}

\noindent Here is the link to the program in MAGMA \cite{magma} used to obtain Proposition \ref{prop2} with a case-by-case analysis:

\bigskip

\url{https://www.dropbox.com/s/lgkk411cwev2mbw/MagmaProgram.docx?dl=0}

\bigskip

\noindent Here are the three main algorithms cited in the paper:
\begin{algorithm}
\caption{
The Hilbert curve $\Gamma$ of a Fano manifold $X$ of index $i$ for
$L:=\frac{r}{\iota_X}(-K_X)$ with $r\in\mathbb{Z}_{\geq 1}$}\label{Alg1}
\begin{algorithmic}[1]
\Require $r, \iota_X, n$
\Ensure $\Gamma$ \If {$n>1, 0<\iota_X\leq n+1$} \State $c_X \gets n-\iota_X+1$
\Procedure{RHilbPolynF}{$\iota_X$} \If {$\iota_X=1$} $\delta (l) \gets 1$
\State \textbf{else} $\delta(l) \gets \prod_{h=1}^{\iota_X-1} (l+h)$
\EndIf \For {$j=0,...,c_X$} $b_j \gets
\frac{1}{\delta(j)} h^0\left(\frac{j}{\iota_X}(-K_X)\right)$ \EndFor \State
$U \gets \texttt{Vandermonde Matrix of } \{0,1,...,c_X\}$ \State
$(a_0,a_1,...,a_{c_X}) \gets (b_0,b_1,...,b_{c_X})\cdot U^{-1}$
\State $R_{(X,L)}(x,y) \gets \left( \sum_{k=0}^{\mu}a_k(ry-mx)^k\right)$
\State \Return $R_{(X,L)}(x,y)$ \EndProcedure \EndIf
\end{algorithmic}
\end{algorithm}  

\bigskip

\begin{algorithm}
\caption{
The Hilbert curve $\Gamma$ of a Fano fibration $\pi :X\to C$
with $\dim X=n$, fiber $F$}\label{Alg2}
\begin{algorithmic}[1]
\Require $F$, $r, \iota_F, n, t$
\Ensure $\Gamma$
\If {$n>1, 0<\iota_F\leq
n$}
\State \texttt{Find $R(x,y)$ with suitable conditions such that}
\If {$t>0$}\State \texttt{Consider
$R(x,y)=R(x-r,y-\iota_F)+t\cdot$RHilbPolynF$(\iota_F)$}
\EndIf
\If {$t<0$}\State \texttt{Consider
$R(x,y)=R(x+r,y+\iota_F)-t\cdot$RHilbPolynF$(\iota_F)$}
\EndIf
\State
$p_{(X,L)}(x,y)\gets R(x,y)\cdot \prod_{k=1}^{\iota_F-1} (ry-\iota_Fx-k)$
\State \Return $\Gamma: p_{(X,L)}(x,y)=0$
\EndIf
\end{algorithmic}
\end{algorithm} 

\bigskip

{\small {
\begin{algorithm}
\caption{
The Hilbert curve $\Gamma$ of a Fano fibration $\pi :X\to Y$ with $m=\dim Y\geq 2$, fiber $F$}\label{Alg3}
\begin{algorithmic}[1]
\Require $A$, $F$, $r, \iota_F, n , m$
\Ensure $\Gamma$
\If {$n>m, 0<\iota_F\leq n-m+1$}
\State $c_F \gets n-m-\iota_F+1$
\State $s \gets 1$
\Repeat
\State $s \gets s+1$
\Until {$sA$ is very ample}
\State $t\gets s^{m-1}(rs+1)^{\frac{m(m-1)}{2}}A^m$
\State \texttt{Find $R_{X_{m-1}}(x,y)$ such that $R_{X_{m-1}}(x,y)=R_{X_{m-1}}(x-r,y-\iota_F)+t\cdot $RHilbPolynF$(\iota_F)$}
\State $M \gets \left( \begin{array}{ c c }
sr+1 & 0 \\
s\iota_F & 1 \\
\end{array}\right)$
\State $\vec{v}\gets\left( -sr , -s\iota_F \right)$
\State $\vec{x}\gets\left( x , y \right)$
\State $j\gets m$
\Repeat
\State $j\gets j-1$
\State \texttt{��Find $R_{X_{j-1}}(x,y)$ such that $R_{X_{j-1}}\big(\vec{x}M)=R_{X_{j-1}}\big(\vec{x}M+\vec{v}\big)+R_{X_j}\big(\vec{x}\big)$��}
\Until $j=1$
\State $p_{(X,L)}(x,y)\gets R_{X_{0}}(x,y)\cdot \prod_{k=1}^{\iota_F-1} (ry-\iota_Fx+k)$
\State \Return $\Gamma: p_{(X,L)}(x,y)=0$
\EndIf
\end{algorithmic}
\end{algorithm}  }}

\vspace{1cm}

\noindent {\bf Acknowledgements}. The first author is a member of G.N.S.A.G.A. of the Italian INdAM. He would like to thank
 the PRIN 2014 Geometry of Algebraic Varieties and the University of Milano for partial support.
 During the preparation of this paper, the second author was partially supported by the National Project Anillo ACT
 1415 PIA CONICYT and the Proyecto VRID N.214.013.039-1.OIN of the University of Concepci\'on.

\end{document}